\documentclass[11pt]{article}
\usepackage{amsfonts}
\usepackage{mathrsfs}

\usepackage[latin1]{inputenc}
\usepackage{amsmath,amssymb,color}
\usepackage{latexsym}
\usepackage{epstopdf}
\usepackage{geometry}   
\usepackage[active]{srcltx}
\usepackage{graphicx}
\usepackage{epstopdf}
\usepackage[
bookmarks=true,         
bookmarksnumbered=true, 
colorlinks=true, pdfstartview=FitV, linkcolor=blue, citecolor=blue,
urlcolor=blue]{hyperref}

\newtheorem{theorem}{Theorem}[section]

\newtheorem{lemma}[theorem]{Lemma}
\newtheorem{proposition}[theorem]{Proposition}

\numberwithin{equation}{section}

\def\R{{\mathbb R}}
\def\E{{{\mathbb E}\,}}

\def\N{{\mathbb N}}

\def\cF{{\cal F}}

\def\Var{{\mathop {{\rm Var\, }}}}
\def\Cov{{\mathop {{\rm Cov\, }}}}
\def\Om{{\Omega}}
\def \eref#1{\hbox{(\ref{#1})}}
\def\nn{{\nonumber}}

\newenvironment{proof}[1][Proof]{\noindent\textit{#1.} }{\hfill \rule{0.5em}{0.5em}}

\begin{document}

\title{Asymptotic behavior for an additive functional of two independent self-similar Gaussian processes}
\date{\today}

\author{David Nualart\thanks{ D. Nualart is supported by the NSF grant
DMS1512891.} \ \
and \ Fangjun Xu\thanks{F. Xu is supported by National Natural Science Foundation of China (Grant No.11401215) and 111 Project (B14019).}\\
}

\maketitle

\begin{abstract}
\noindent  We derive the asymptotic behavior for an additive functional of two independent self-similar Gaussian processes when their intersection local time exists, using the method of moments.


\end{abstract}

\section{Introduction}
The aim of this paper is to show a central limit theorem for an additive functional of two independent identically distributed Gaussian processes.
Let  $X= \{ X_t, t\geq 0\} $ be a $d$-dimensional  centered Gaussian  process.
We assume that $X$ is  $H$-self-similar for some parameter  $H\in (0,1)$ such that $Hd<2$. That is, $X$ satisfies the scaling property
\begin{equation*}
\big\{X(ct), t\ge 0\big\}\overset{\mathcal{L}}{=}\big\{c^HX(t), t\ge 0\big\}.   \label{intr1}
\end{equation*}
for all $c>0$.

We will denote by $X^{(1)}$ and $X^{(2)}$ two independent copies of  $X$.  
For any rectangle $E \subset \R^2_+$, we will denote by $L(x,E)$   the local time  of the  two-parameter process $Z=\{ Z(t,s)= X^{(1)}_t -X^{(2)}_s, s,t \ge 0\}$,   defined,  if it exists, as the density of the occupation measure, that is,
\[
\int_{\R^d}   \varphi(x) L(x,E) dx = \int_E  \varphi(Z(t,s) )dt ds,
\]
for any measurable and bounded function $\varphi :\R^d  \rightarrow \R$, see \cite{wu_xiao} for fractional Brownian motion case.   When $E=[0,t_1]\times[0,t_2]$, $I(t_1, t_2)=L(0,E)$ is the intersection local time of $X^{(1)}$ and $X^{(2)}$ that can be defined as
\[
I(t_1, t_2)=\int^{t_1}_0\int^{t_2}_0 \delta(X^{(1)}_u-X^{(2)}_v)\, du\, dv,
\]
where $\delta$ is the Dirac delta function, see \cite{nol} for fractional Brownian motion case.

  Given an any integrable function $f:\R^d\to\R$, we are interested in the asymptotic behavior as $n\rightarrow \infty$ of the additive functional
 $ \int^{nt_1}_0\int^{nt_2}_0 f(X^{(1)}_u-X^{(2)}_v) du dv$.   Suppose that  for any rectangle $E=[0,t_1]\times[0,t_2]$ the  local time  $L(x,E)$   exists and it is continuous at $x=0$.   We also assume $Hd<2$.
  Then,
   one can easily show the following convergence in law in the space $C([0,\infty)^2)$, as $n$ tends to infinity
\begin{equation*} \label{e.0.1}
\left( n^{Hd-2}\int^{nt_1}_0\int^{nt_2}_0 f(X^{(1)}_u-X^{(2)}_v)\, du\, dv,\; t_1, t_2\ge 0 \right) \overset{\mathcal{L}}{\longrightarrow } \left( I(t_1,t_2) \int_{\R^d} f(x)\, dx,\;  t_1, t_2\ge 0 \right).
\end{equation*}
In fact, letting $E=[0,t_1]\times[0,t_2]$ and  using the self-similarity and the existence and continuity of the intersection local time at zero,  we get
\begin{align*}
n^{Hd-2}\int^{nt_1}_0\int^{nt_2}_0 f(X^{(1)}_u-X^{(2)}_v)\, du\, dv
&\overset{\mathcal{L}}{=}   n^{Hd}\int^{t_1}_0\int^{t_2}_0 f\big(n^H (X^{(1)}_u-X^{(2)}_v) \big)\, du\, dv \\
&=                                           n^{Hd} \int_{\R^d}f(n^H x)L(x,E)\, dx\\
&=                                            \int_{\R^d} f(x)L(\frac{x}{n^H},E)\, dx\\
&\overset{\mathcal{L}}{\longrightarrow } I(t_1,t_2) \int_{\R^d} f(x)\, dx.
\end{align*}

If we assume that $\int_{\R^d} f(x)\, dx=0$, then $n^{Hd-2} \int^{nt_1}_0\int^{nt_2}_0 f(X^{(1)}_u-X^{(2)}_v)\, du\, dv$ converges to $0$. It is interesting to know if there is a $\alpha>Hd-2$ such that $n^{\alpha}\int^{nt_1}_0\int^{nt_2}_0  f(X^{(1)}_u-X^{(2)}_v)\, du\, dv$ converges  to a nonzero process. 
It turns out that the natural choice is $\alpha= \frac {Hd-2}2$, the convergence is in distribution and the limit is a   mixture of Gaussian laws. The purpose of this paper is to show the corresponding  functional  central limit theorem for a large class of Gaussian processes $X$.  We will first describe the class of processes we will consider.

Along the paper $X=\{X_t, t\ge 0\}$ is a $d$-dimensional  $H$-self-similar  centered Gaussian stochastic process whose components $X^\ell$, $1\le \ell \le d$, are independent and identically distributed.   We will assume that the process $X^\ell$ satisfies the following hypotheses:

\smallskip
\noindent
{\bf (H1)}   {\it  Nondeterminism property}: 
there exists a positive constant $\kappa$ depending only on $n$ and  $H$, such that for any $0=t_0<t_1 <\dots < t_n $ and $x_i \in \mathbb{R}$, $1\le i \le n$, we have
\begin{equation} \label{eku1}
\mathrm{Var} \Big( \sum_{i=1}^n x_i   (X^\ell_{t_i} -X^\ell_{t_{i-1}}) \Big)\geq \kappa  \sum_{i=1}^n x_i^2 (t_i -t_{i-1})^{2H}.
\end{equation}

\smallskip
\noindent
{\bf (H2)}  {\it Bounds on the variance of increments}:  There exist positive constants $\gamma_0\geq 1$, $\alpha_2>0$ and nonnegative decreasing functions $\phi_{i}(\varepsilon)$ on $[0,1/\gamma_0]$ with $ \lim\limits_{\varepsilon\to 0}\phi_{i}(\varepsilon)=0$, for $i=1,2,$ such that
\[
0\leq h^{2H}(\alpha_2-\phi_{1}(\frac{h}{t}))\leq \mathrm{Var}(X^\ell_{t+h}-X^\ell_t)\leq h^{2H}(\alpha_2+\phi_{2}(\frac{h}{t}))
\]
for all $h\in[0,t/\gamma_0]$.

\smallskip
\noindent
{\bf (H3)}    {\it Bounds on the covariance of increments on disjoint intervals}:   there exists a nonnegative decreasing function $\beta(\gamma): (1,\infty) \rightarrow \mathbb{R}$ with $\lim\limits_{\gamma\to\infty}\beta(\gamma)=0$, such that,
for any $0<t_1<t_2<t_3<t_4<\infty$   such that
  $\frac{\Delta t_2}{\Delta t_4}\leq \frac{1}{\gamma}$ or $\frac{\Delta t_2}{\Delta t_4} \geq \gamma$ or $\max\left( \frac{\Delta t_2}{\Delta t_3}, \frac{\Delta t_4}{\Delta t_3}\right) \leq \frac{1}{\gamma}$, then
 \[
\big|\E\big(X^\ell_{t_4}-X^\ell_{t_3}\big)\big(X^\ell_{t_2}-X^\ell_{t_1}\big)\big|\leq  \beta(\gamma)\, \left[\E\big(X^\ell_{t_4}-X^\ell_{t_3}\big)^2\right]^{\frac{1}{2}}\left[\E\big(X^\ell_{t_2}-X^\ell_{t_1}\big)^2 \right]^{\frac{1}{2}},
\] 
where $\Delta t_i=t_i-t_{i-1}$ for $i=2,3,4$.

From results in \cite{sxy}, we see that the following Gaussian processes satisfy  the above hypotheses:

\noindent
(i) {\it Bifractional Brownian motion}. The covariance function of this process  is given by
\[
\E(X^\ell_t X^\ell_s)=2^{-K_0}[(t^{2H_0}+s^{2H_0})^{K_0}-|t-s|^{2H_0K_0}],
\]
where $H_0\in(0,1)$ and $K_0\in(0,1]$.  
See \cite{Houdre, RussoTudor06} for the main properties of this process, and note that $K=1$ gives the classic fractional Brownian motion case with Hurst parameter $H=H_0$.
Hypotheses {\bf (H1)}-{\bf (H3)} hold with $H=H_0K_0$.  
 
\noindent
(ii) {\it Subfractional Brownian motion}.  The covariance function of this process  is given by
\[
\E(X^\ell_t X^\ell_s)=t^{2H}+s^{2H}-\frac{1}{2}[(t+s)^{2H}+|t-s|^{2H}],
\]
where $H\in(0,1)$.  This Gaussian process has been studied in \cite{Bojdecki, Chavez} and it satisfies   {\bf (H1)}-{\bf (H3)}.
 
 \noindent
 (iii) More generally, the  Gaussian self-similar processes considered in  \cite{HN} satisfy Hypotheses {\bf (H2)} and {\bf (H3)} in the particular case $\alpha =2\beta=2H$.

 It can be proved  that if the components of $X$ satisfy Hypothesis 
{\bf (H1)}  and $Hd<2$, then  local time  $L(x,E)$   exists and  is continuous in $x$.    Indeed, this property can be established using the arguments of the  proof of Theorem 8.1 in  \cite{Be}, together with the lower bound for the variance based on the nondeterminism property obtained in Subsection 3.1 below.

In order to  formulate our result  we introduce the
following space of functions. Fix a number $\beta\in(0,2)$, define
\[
H^{\beta}_0=\Big\{f\in L^1(\R^d):\, \int_{\R^d} |f(x)||x|^{\beta}\, dx<\infty \quad\text{and}\quad \int_{\R^d} f(x)\, dx=0 \Big\}. \label{beta}
\]
We will denote by $\widehat{f}$ the Fourier transform of a function $f$.

The next theorem is the main result of this paper, which is a functional version of the central limit theorem  in the case where $X=B^H$ is a $d$-dimensional fractional Brownian motion (fBm)  with Hurst parameter $H$, proved in \cite{nx2}.
\begin{theorem}  \label{thm1} Suppose $\frac{2}{d+2}<H<\frac{2}{d}$ and $f\in H^{\frac{2}{H}-d}_0$. Under hypotheses {\bf (H1)-(H3)}, then
\[
n^{\frac{Hd-2}{2}} \int^{nt_1}_0\int^{nt_2}_0  f(X^{(1)}_u-X^{(2)}_v)\, du\, dv\overset{\mathcal{L}}{\longrightarrow} \sqrt{D_{f,H,d}}\; \Lambda(t_1, t_2),
\]
in $C([0,\infty)^2)$, as $n\to\infty$, where, conditionally  on $X^{(1)}$ and $X^{(2)}$, $\Lambda(t_1, t_2)$ is a two-parameter centered Gaussian process with covariance function
\begin{align*}
\E\left[\Lambda\left(t_1, t_2\right) \Lambda\left(s_1, s_2\right)\right]= I(t_1\wedge s_1, t_2\wedge s_2)
\end{align*}
and 
\[
D_{f,H,d}=\left(2\left(\frac{2}{\alpha_2}\right)^{\frac{1}{2H}}\, \Gamma\Big(\frac{2H+1}{2H}\Big)\right)^2\left(\frac{1}{(2\pi)^d}\int_{\R^d} |\widehat{f}(x)|^2|x|^{-\frac{2}{H}}\, dx\right).
\]
\end{theorem}

Let us discuss the role of our hypotheses in this theorem. The nondeterminism property {\bf (H1)} guarantees the existence and  continuity of the self-intersection local time. 
 Hypothesis {\bf (H2)} means that  $\mathrm{Var}(X^\ell_{t+h}-X^\ell_t)$ behaves as $\alpha_2 h^{2H}$ as $h\rightarrow 0$ and  is well tailored for 
 Gaussian processes with nonstationary increments.
   The independent normal random phenomenon appearing in the second-order limit law for functionals of Gaussian processes is caused by the increments in Hypothesis {\bf (H2)}.  On the other hand, Hypothesis {\bf (H3)} characterizes the covariance of increments having no contribution to the limiting distribution.

In the Brownian motion case ($H=\frac{1}{2}$ and $d=3$), Theorem \ref{thm1} can be proved using a theorem by Weinryb and Yor \cite{weinryb_yor}. Note that the constant $D_{f,H,d}$ is finite for any $H>\frac{2}{d+2}$.    

In \cite{hnx}, we proved the following central limit theorem when $X=B^H$, assuming $Hd<1$
\begin{align*}
  \Big(  n^{\frac{Hd-1}{2}}\int_{0}^{nt} f(B^H_s)\, ds\,,
  \ t\ge 0\Big) \ \ \overset{\mathcal{L}}{\longrightarrow } \ \
  \Big(  \sqrt{C_{H,d}} \, \| f\|_{\frac{1}{H}-d}\, W (L_{t}(0))\,, t\ge
  0\Big),
\end{align*}
where $L_t(0)$ is the local time at $0$ of $B^H$ and $W$ is a Brownian motion independent of $B^H$.
Using the methodology we  develop here, the above result could be extended to a centered  $d$-dimensional self-similar Gaussian  process $X$ satisfying hypotheses {\bf (H1)-(H3)}, that is,
\begin{align*} \label{hnx}
  \Big(  n^{\frac{Hd-1}{2}}\int_{0}^{nt} f(X_s)\, ds\,,
  \ t\ge 0\Big) \ \ \overset{\mathcal{L}}{\longrightarrow } 
  \Big(  \sqrt{C_{f,H,d}}\; W (L_{t}(0))\,, t\ge
  0\Big),
\end{align*}
where 
\[
C_{f,H,d}=\left(2\left(\frac{2}{\alpha_2}\right)^{\frac{1}{2H}}\, \Gamma\Big(\frac{2H+1}{2H}\Big)\right)\left(\frac{1}{(2\pi)^d}\int_{\R^d} |\widehat{f}(x)|^2|x|^{-\frac{1}{H}}\, dx\right).
\]

So it is natural to guess that, for $N$ independent  copies  $X^{(1)}, \dots, X^{(N)}$ of  a  centered $d$-dimensional self-similar Gaussian  process $X$ satisfying hypotheses {\bf (H1)-(H3)},  assuming
 $\frac{N}{d+2}<H<\frac{N}{d}$ and $f\in H^{\frac{N}{H}-d}_0$, the following result should hold 
\begin{align*}
 & \Big(  n^{\frac{Hd-N}{2}}\int_{0}^{nt_1}\dots \int_{0}^{nt_N}  f\big(\sum^N_{i=1} X^{(i)}_{s_i}\big)\, ds\,,\;
  t_1\ge 0,\dots, t_N\ge 0\Big) \\
  &\qquad\qquad\qquad\qquad \overset{\mathcal{L}}{\longrightarrow } 
  \Big(  \sqrt{D_{f,N,H,d}}\;  \Lambda(t_1,\dots, t_N)\,,\; t_1\ge
  0,\dots, t_N\ge 0\Big),
\end{align*}
where 
\[
D_{f,N,H,d}=\left(2\left(\frac{2}{\alpha_2}\right)^{\frac{1}{2H}}\, \Gamma\Big(\frac{2H+1}{2H}\Big)\right)^N\left(\frac{1}{(2\pi)^d}\int_{\R^d} |\widehat{f}(x)|^2|x|^{-\frac{N}{H}}\, dx\right)
\]
and, conditionally on $X^{(1)},\dots, X^{(N)}$, $\Lambda(t_1,\dots, t_N)$ is a two-parameter centered Gaussian process having covariance function
\begin{align*}
\E\left[\Lambda\left(t_1,\dots, t_N\right) \Lambda\left(s_1,\dots, s_N\right)\right]= I(t_1\wedge s_1,\dots, t_N\wedge s_N)
\end{align*}
with $I(t_1,\dots, t_N)$ being the local time of the multiparameter process $\sum^N_{i=1} X^{(i)}(u_i)$ at the origin on the time rectangle $\prod^N_{i=1}[0,t_i]$. The interested readers could prove this easily by using our methodology.

This paper can be viewed as an extension of the result in \cite{hnx}. The limit here is different from that in \cite{hnx}, which, conditionally on $X^{(1)}$ and $X^{(2)}$, is a two-parameter Gaussian process. To prove our main result Theorem \ref{thm1}, we use Fourier analysis and the method of moments. Some techniques in \cite{hnx} will be applied, but new ideas are also needed. For example, we use the paring technique introduced in \cite{sxy} to prove the convergence of moments. On the other hand, this paper sheds light on proving asymptotic behavior of additive functionals of multi-parameter processes or random fields when intersection local time or local time exists.

A second-order result for two independent Brownian motions in the critical case $d=4$ and $H=\frac{1}{2}$ was proved by Le Gall \cite{LeGall}.  General asymptotic results for additive functionals of $k$ independent Brownian motions were obtained by Biane \cite{biane}.  A first-order result for two independent fBms in the critical case $Hd=2$ with $H\leq 1/2$ was given in \cite{bx}. Recently, both first-order and second-order results for two independent general Gaussian processes in the critical case $Hd=2$ were proved in \cite{sxy}, where the functional limit theorems are still unknown.  Extensions to functionals of $k$ independent Gaussian processes in the critical case $Hd=k$ are also mentioned in \cite{sxy}.

After some preliminaries in Section 2, Section 3 is devoted to the proof of Theorem 1.1,
based on the method of moments. Throughout this paper, if not mentioned otherwise, the letter $c$, 
with or without a subscript, denotes a generic positive finite
constant whose exact value is independent of $n$ and may change from
line to line. We use $\iota$ to denote $\sqrt{-1}$.

\section{Preliminaries}
Let $\left\{X_t= (X^1_t,\dots,
X^d_t), t\geq 0\right\} $ be a $d$-dimensional centered Gaussian process defined on some probability space $(\Om, \cF, P)$, whose components are independent, identically distributed and satisfy Hypotheses {\bf (H1)}-{\bf (H3)}, with $Hd<2$.

We will denote by $X^{(1)}$ and $X^{(2)}$ two independent copies of $X$.  
 Conditionally on $X^{(1)}$ and $X^{(2)}$, let $\Lambda(t_1, t_2)$ be a two-parameter centered Gaussian process with covariance function 
\[
\E\left[\Lambda\left(t_1, t_2\right) \Lambda\left(s_1, s_2\right)\right]= I(t_1\wedge s_1, t_2\wedge s_2),
\] 
where $I(t_1,t_2)$ is the intersection  local time of the processes $X^{(1)}$ and $X^{(2)}$ on the rectangle $[0,t_1]\times [0,t_2]$.
 
 Given any rectangle $E=(a, b] \times (c, d]$ in $\R^2_+$, we use $\Delta_E \Lambda$ to denote the increment of $\Lambda$ on $E$. That is, 
\[
\Delta_E \Lambda=\Lambda (b,d) - \Lambda (a,d)-\Lambda (b,c)+\Lambda (a,c).
\] 
The next lemma gives a formula for the moments of increments of the process  
$\left\{\Lambda(t, s), t,s\ge 0\right\}$ on disjoint rectangles.
\begin{lemma} \label{lema2.1} Fix a finite number of disjoint rectangles 
$E_i =(a_i, b_i] \times (c_i, d_i] $,
where $i=1,\dots, N$. Consider a multi-index $\mathbf{m} = (m_1, \dots, m_N)$, where $m_i \geq 1$ and 
$1 \leq  i \leq N$. Then
\begin{eqnarray}
&&\E \Big(\prod_{i=1}^N \big[\Delta_{E_i}\Lambda\big]^{m_i} \Big) \label{2.1} \\
&=&\begin{cases} 
\Big(\prod\limits_{i=1}^N  \frac {m_i! }{2^{ \frac {m_i}2} (2\pi )^{\frac{m_id}{4}} (m_i/2)!}\Big)   \int_{\prod\limits_{i=1}^N E_i^{\frac{m_i}2}} (\det A (w, \tau))^{-\frac{1}{2}}  dw\, d\tau
&\hbox{if all $m_i$ are even}\\
0 & \hbox{otherwise}, 
\end{cases}   \notag
\end{eqnarray}
where $A(w,\tau)$ is the covariance matrix of the Gaussian random vector 
\[
\Big(X^{(1)}_{w^i_k}-X^{(2)}_{\tau^i_k}: 1\leq i\leq N\; \text{and}\; 1\leq k\leq \frac{m_i}{2}\Big).
\]
\end{lemma}

\begin{proof} This follows from the definition of the two-parameter process $\Lambda(t_1,t_2)$.
\end{proof}

\medskip
We next claim that the law of the random vector  $\big(\Delta_{E_i}\Lambda: 1\le i\le N\big)$ is determined by
the moments computed in Lemma \ref{lema2.1}. This is a consequence of the following estimates. Fix an even integer $n=2k$, and set
$D_k=[0,t_1]^{k}\times[0,t_2]^{k}$. For any $(u,v) \in D_k$, let  $A_k(u,v)$ be the covariance matrix of the Gaussian random vector 
\[
\Big(X^{(1)}_{u_1}-X^{(2)}_{v_1},\, X^{(1)}_{u_2}-X^{(2)}_{v_2},\, \dots,\, X^{(1)}_{u_k}-X^{(2)}_{v_k}\Big).
\]
Then  the  nondeterminism property (\ref{eku1}) implies
\begin{equation}
     \big(\det A_k(u,v) \big)^{-\frac{1}{2}}  \le c_3 \prod^k_{i=1}\big(u_{\sigma(i)}-u_{\sigma(i-1)}\big)^{-\frac{Hd}{2}}\prod^k_{i=1}\big(v_{\pi(i)}-v_{\pi(i-1)}\big)^{-\frac{Hd}{2}}, \label{det}
\end{equation}
where $\sigma$ and $\pi$ are two permutations of indices $i=1,\dots,k$ such that $u_{\sigma(1)}<u_{\sigma(2)}<\dots<u_{\sigma(k)}$ and $v_{\pi(1)}<v_{\pi(2)}<\dots<v_{\pi(k)}$, and by convention $u_{\sigma(0)}=v_{\pi(0)}=0$.
As a consequence of \eref{2.1} and \eref{det},  
\begin{align*}
\E\big[\Lambda(t_1,t_2)\big]^n
& \le  c_4\, n!  \int_{D_k}\big(\det A_k(u,v) \big)^{-\frac{1}{2}}\, du\, dv \\
& \le c_5\, n!  \bigg(\int_{\{0<w_1 + \dots + w_k <t_1\}} \prod_{i=1}^k w_i^{-\frac{Hd}{2}}\, dw\bigg)\bigg(\int_{\{0<w_1 + \dots + w_k <t_2\}} \prod_{i=1}^k w_i^{-\frac{Hd}{2}}\, dw\bigg) \\
& =  c_6\, n! (t_1t_2)^{k(1-\frac{Hd}{2})}\bigg(\frac
{\Gamma^{k}(1-\frac{Hd}{2})}{\Gamma\big(k(1-\frac{Hd}{2})+1\big)}\bigg)^2.
\end{align*}
Therefore,  $\E\big[\Lambda(t_1,t_2)\big]^n $ is bounded by $c^k n!/
\Big(\Gamma\big(k(1-\frac{Hd}{2})+1\big)\Big)^2$, and this easily implies the desired characterization of the law of the increments of the two-parameter process
 $\left\{\Lambda(t,s): t\ge 0, s\ge 0\right\}$ on disjoint rectangles by its moments.

We will make use of the following property of the space $H_0^{\frac 2H-d}$. For any $f\in  H_0^{\frac 2H-d}$ we define
\begin{equation} \label{enf}
N(f)= \int_{\mathbb{R}^d} |f(z)| \Big( 1+ |z| ^{\frac 2H -d} \Big)\, dz.
\end{equation}

\begin{lemma} \label{lem2.3}
Let $f\in H_0^{\frac 2H-d}$. For any $\alpha  \in \big[0,  (\frac  2H -d) \wedge 1\big]$ and any $ x,y\in \mathbb{R}^d$,
\begin{equation*} \label{w1}
| \widehat{f} (x) -\widehat{f} (y) | \le (2\pi) ^{-d}\, N(f)\, |x-y|^\alpha,
\end{equation*}
where $N(f)$ is given in (\ref{enf}).
\end{lemma}
\begin{proof}
We can write
\[
|\widehat{f}(x)-\widehat{f}(y)|
=\frac{1}{(2\pi)^d}\int_{\R^d} |e^{\iota x\cdot z}-e^{\iota y\cdot z}|| f(z)|\, dz
\leq \frac{1}{(2\pi)^d}  |x-y|^{\alpha}\int_{\R^d} |z|^{\alpha} |f(z)|\, dz.
\]
Then, if $|z| \leq 1$, we estimate $|z|^\alpha$ by $1$, and if $|z| \geq 1$ we use
$|z|^\alpha\le |z|^{\frac 2H -d}$.
\end{proof}

\section{Estimates and convergence of even moments}

Fix $a_1, a_2, b_1,b_2\ge 0$ with $a_1<b_1$ and $a_2<b_2$.  We consider the random variable
\[
F_{n}(a_1,b_1; a_2, b_2)=n^{\frac{Hd-2}{2}} \int^{nb_1}_{na_1}\int^{nb_2}_{na_2} f(X^{(1)}_u-X^{(2)}_v)\, du\, dv.
\]
Using the Fourier transform of $f$, denoted by $\widehat{f}$, we can write, for any integer $m\ge 1$,
\begin{align}
&\E[F_{n}(a_1,b_1; a_2, b_2)^m] \notag\\
&=n^{m\frac{Hd-2}{2}} \int_{E_n^m} 
\E \Big( \prod _{i=1}^m f(X^{(1)}_{u_i}-X^{(2)}_{v_i})\Big) \, du\, dv \notag \\
&=\frac{n^{m\frac{Hd-2}{2}} }{(2\pi)^{md}}\,
 \int_{\R^{md}} \int_{E_n^m}   \Big( \prod _{i=1}^m \widehat{f} (y_i)   \Big)  \exp \bigg( -\frac 12 {\rm Var} \Big(\sum_{i=1}^m y_i \cdot (X^{(1)}_{u_i}-X^{(2)}_{v_i})\Big)\bigg) \, du\, dv\, dy \notag\\
&=\frac{n^{m\frac{Hd-2}{2}} }{(2\pi)^{md}}\, \int_{\R^{md}} \int_{E_n^m}   \Big( \prod _{i=1}^m \widehat{f} (y_i)   \Big) \notag  \\
&\quad \times
      \exp \bigg( -\frac 12 {\rm Var} \Big(\sum_{i=1}^m y_i \cdot X_{u_i}\Big)
      -\frac 12 {\rm Var} \Big(\sum_{i=1}^m y_i \cdot X_{v_i}\Big)\bigg) \,du\, dv\,  dy, \label{e2}
\end{align}
where $E_n=[na_1,nb_1] \times [na_2,nb_2]$ and we have used that $X^{(1)}$ and $X^{(2)}$ are independent copies of $X$.

\medskip
In order to estimate this expectation we proceed in several steps.

\subsection{Lower bound for the variance using nondeterminism} 
 In order to apply the nondeterminism property we would like to replace the integral on the rectangles $[na_1,nb_1]^m$ and $[na_2,nb_2]^m$ by $(m! )^2$ times the integrals over the associated simplexes $\{na_1<u_1< \dots <u_m<nb_1\}$ and $\{na_2<v_1 <\dots < v_m<nb_2 \}$. Unfortunately, we cannot do this because both exponential factors are linked through the same coefficients $y_i$'s. To overcome this difficulty, we will make use of the Cauchy-Schwarz inequality after some rearrangements of the terms.

 Let $\mathscr{P}$ be the set  of all permutations of $\{1,2,\dots, m\}$.  For $\ell=1,2$, set
 \begin{equation*} \label{simplex}
 D^m_{a_{\ell}, b_{\ell}}=\Big\{ u\in [na_{\ell},nb_{\ell} ]^m:na_{\ell}<u_1<\dots <u_m <nb_{\ell} \Big\},
 \end{equation*}
and  define
 \[
I_{a_1, b_1}(y)= \int_{ D^m_{a_1, b_1} } \exp\bigg( -\frac 12 \mathrm{Var} \Big( \sum_{i=1}^{m} y_i \cdot X_{u_i}  \Big)\bigg)\, du
 \]
and  \[
I_{a_2, b_2}^\sigma(y)= \int_{ D^m_{a_2, b_2} } \exp \bigg(-\frac 12 \mathrm{Var} \Big( \sum_{i=1}^{m} y_i \cdot X_{u_{\sigma(i)}}  \Big)\bigg)\, du
\]
for any $\sigma\in\mathscr{P}$. 
Then,   expression (\ref{e2}) for the moment of order $m$ can be also written as
\begin{equation*} \label{w2}
 \E[F_{n}(a_1,b_1; a_2, b_2)^m]= \frac{m!}{(2\pi)^{md}}\, n^{m\frac{Hd-2}{2}}  \sum_{\sigma \in \mathscr{P}}  \int_{\R^{md}} \Big( \prod _{i=1}^m  \widehat{f} (y_i)   \Big)
I_{a_1, b_1}(y)  I_{a_2, b_2}^\sigma(y)\, dy.
\end{equation*}
 By the Cauchy-Schwarz inequality, we  obtain
 \begin{align*}   \notag
 &\bigg|  \int_{\R^{md}} \Big( \prod _{i=1}^m  \widehat{f} (y_i)   \Big)
I_{a_1, b_1}(y)  I_{a_2, b_2}^\sigma(y)\, dy  \bigg|  
\le 
\bigg(\int_{\R^{md}} \Big( \prod _{i=1}^m  |\widehat{f} (y_i)  | \Big)
\big(I_{a_1, b_1}(y)\big)^2\,  dy     \bigg)^{\frac 12} \\
& \qquad\qquad\qquad\qquad\qquad\qquad\qquad\qquad\qquad\qquad \times  
 \bigg(\int_{\R^{md}} \Big( \prod _{i=1}^m  |\widehat{f} (y_i)  | \Big)
\big(I_{a_2, b_2}^\sigma(y) \big)^2\,  dy \bigg)^{\frac 12}.
\end{align*}
 Taking into account that $\prod _{i=1}^m| \widehat{f} (y_i) |  $ is a symmetric function of the $y_i$'s, 
 the second factor in the above expression does not depend on $\sigma$ and we obtain
 \begin{align*}   \notag
 \left| \E[F_{n}(a_1,b_1; a_2, b_2)^m] \right|  
&\le 
\frac{(m!)^2}{(2\pi)^{md}}\, n^{m\frac{Hd-2}{2}} \bigg(\int_{\R^{md}} \Big( \prod _{i=1}^m  |\widehat{f} (y_i)  | \Big)
\big(I_{a_1, b_1}(y)\big)^2\,  dy     \bigg)^{\frac 12} \\
& \qquad\qquad\qquad\qquad\qquad \times  \label{eq1}
 \bigg(\int_{\R^{md}} \Big( \prod _{i=1}^m  |\widehat{f} (y_i)  | \Big)
\big(I_{a_2, b_2}(y) \big)^2\,  dy \bigg)^{\frac 12}.
\end{align*}

For $\ell=1,2$, making the change of variables $x_i =\sum_{j=i}^m y_j$ (with the convention $x_{m+1} =0$) and using the notation
$\Delta X_{u_i} =X_{u_i} - X_{u_{i-1}}$ and  $\Delta X_{v_i} =X_{v_i} - X_{v_{i-1}}$, we can write
 \begin{eqnarray*}
&&\int_{\R^{md}} \Big( \prod _{i=1}^m  |\widehat{f} (y_i)  | \Big)
\big(I_{a_{\ell}, b_{\ell}}(y)\big)^2\,  dy \le  \int_{\R^{md}} \int_{D^m_{a_{\ell}, b_{\ell}}\times D^m_{a_{\ell}, b_{\ell}}}
\prod _{i=1}^m| \widehat{f} (x_i-x_{i+1} ) | \\
&& \quad  \times 
\exp\bigg( -\frac 12 \mathrm{Var} \Big( \sum_{i=1}^{m} x_i \cdot  \Delta X_{u_i}  \Big)  -\frac 12 \mathrm{Var} \Big( \sum_{i=1}^{m} x_i \cdot  \Delta X_{v_i} \Big) \bigg) \,du\, dv\, dx.
\end{eqnarray*}
Applying the  nondeterminism property (\ref{eku1}) and making the change of variables $s_1 =u_1$, $r_1=v_1$, $s_i =u_i-u_{i-1}$,  and  $r_i =v_i- v_{i-1}$, for $2\le i  \le m$, we obtain
 \begin{eqnarray*}
&& \int_{\R^{md}} \Big( \prod _{i=1}^m  |\widehat{f} (y_i)  | \Big)
\big(I_{a_{\ell}, b_{\ell}}(y)\big)^2\,  dy \le  \int_{\R^{md}} \int_{[0,n(b_{\ell}-a_{\ell})]^{2m}} 
\bigg( \prod _{i=1}^m| \widehat{f} (x_i-x_{i+1} ) |  \bigg)\\
&& \qquad\qquad\qquad\qquad \times 
\exp\bigg( -\frac {\kappa} 2   \sum_{i=1}^{m} |x_i|^{2} (s_i^{2H} + r_i^{2H})\bigg) \,ds\, dr\, dx.
\end{eqnarray*}

\medskip

\subsection{Chaining argument} 
The next step consists of using the chaining argument introduced in the reference \cite{nx1}. The main idea is to replace each product $\widehat{f}(x_{2i-1}-x_{2i})\widehat{f}(x_{2i}-x_{2i+1})$ by   
$\widehat{f}( -x_{2i})\widehat{f}(x_{2i} )= |\widehat{f}(  x_{2i})|^2 $. Then, by Lemma \ref{lem2.3},  the differences  $\widehat{f}(x_{2i-1}-x_{2i})- \widehat{f}( -x_{2i})$ and
$\widehat{f}(x_{2i}-x_{2i+1})- \widehat{f}(  x_{2i})$ are bounded by constant multiples of $|x_{2i-1}|^\alpha$ and $|x_{2i+1}|^\alpha$, respectively, for any
$0\le \alpha \le (\frac{2}{H}-d)\wedge1$.   We are going to make these substitutions recursively.  We can write
 \begin{align*}
\prod _{i=1}^m| \widehat{f} (x_i-x_{i+1} ) | & =  | \widehat{f} (x_1-x_{2} )  -\widehat{f} (-x_{2} ) +
  \widehat{f} (-x_{2}) | | \widehat{f} (x_2-x_{3} )  -\widehat{f} (x_{2 }) +
  \widehat{f} (x_{2}) |  \\
  & \times | \widehat{f} (x_3-x_{4} )  -\widehat{f} (-x_{4 }) +
 \widehat{f} (-x_{4}) | | \widehat{f} (x_4-x_{5} )  -\widehat{f} (x_{4 }) +
\widehat{f} (x_{4}) |  \times \dots \\
  & = \prod _{i=1}^m \Big| \widehat{f} (x_i-x_{i+1} )  -\widehat{f} \big((-1)^ix_{2 \lfloor \frac {i+1}2 \rfloor }\big) +
  \widehat{f} \big((-1)^i x_{2 \lfloor \frac {i+1}2 \rfloor } \big) \Big|,
  \end{align*}
  where $ \lfloor \frac {i+1}2 \rfloor  $ denotes the integer part of $\frac {i+1}2$. Note that $|\widehat{f}(x)|=|\widehat{f}(-x)|$. Then,
  \[
    \prod _{i=1}^m| \widehat{f} (x_i-x_{i+1} ) |
   \le \sum_{k=1}^m I_k,
   \]
  where
  \[
  I_k=\Big( \prod _{j=1}^{k-1} \big| \widehat{f}(x_{2 \lfloor \frac {j+1}2 \rfloor }) \big| \Big)
  \big| \widehat{f} (x_k-x_{k+1} )  -\widehat{f}((-1)^k x_{2 \lfloor \frac {k+1}2 \rfloor })\big|  \prod_{j=k+1}^m \big| \widehat{f} (x_j-x_{j+1} )\big|
  \]
    for $k=1, 2, \dots, m-1$,
  and
  \[
  I_m= \Big( \prod _{j=1}^{m-1} \big| \widehat{f}(x_{2 \lfloor \frac {j+1}2 \rfloor }) \big| \Big) \big|\widehat{f} (x_m)\big|.
  \]
 In this way, we obtain the decomposition
\begin{equation*} \label{eq3}
 \big|\E[F_{n}(a,b; a, b)^m]\big|\le \frac{(m!)^2}{(2\pi)^{md}}\,    \sum_{k=1}^m A_{k,m},
 \end{equation*}
where
\[
A_{k,m}=n^{m\frac{Hd-2}{2}}  \int_{\R^{md}} \int_{[0,n(b-a)]^{2m}} 
I_k\, \exp\Big( -\frac {\kappa} 2   \sum_{i=1}^{m} |x_i|^{2}( s_i^{2H}+r_i^{2H}) \Big) \, ds\,dr\, dx.
\]

Fix a nonnegative constant $\lambda$ such that
\begin{equation*}  \label{gamma}
 \lambda<
\begin{cases}
\frac{2-Hd }{2} & \text{if } 2-Hd\leq H; \\
\frac{2H-2+Hd}{2} & \text{if } H< 2-Hd< 2H.
\end{cases}
\end{equation*}
The estimation of each term $A_{k,m}$ is given in the next lemma.
\begin{lemma} \label{chain}  There exists a positive constant $c$ such that  for   $k=1,2,\dots,m-1$ and also for $k=m$ if $m$ is odd,  
\begin{equation} \label{eku2}
A_{k,m}\leq c\, N(f)^{m}\,   (b-a)^{\frac{m(2-Hd)}{2}-\lambda}\, n^{-\lambda},
\end{equation}
and if $m$ is even,
\begin{equation} \label{eku3}
A_{m,m} \le c\,  N(f)^m\,  (b-a)^{\frac{m(2-Hd)}{2}}.
\end{equation}
\end{lemma} 
\begin{proof}  When $a=0$, $b=t$ and $\lambda>0$, the above estimates were obtained in Lemma 3.1 of \cite{nx2} in the case where $X=B^H$ is a $d$-dimensional fBm with Hurst parameter $H\in (0,1)$. Using similar arguments as in \cite{nx2}, we could obtain our results.
\end{proof}

\medskip

\subsection{Paring Technique}  

The estimates obtained by using the chaining argument play a critical role when deriving  limit theorems for an additive  functional of the fractional Brownian motion (fBm), see \cite{nx1, xu}. In fact,  using some estimates on the covariance of increments of the fBm on disjoint intervals (see Lemma 2.4 in \cite{sxy}), 
the convergence of even moments in \cite{hnx, nx1} can be easily obtained  applying the method given in the proof of Proposition 4.2 of \cite{xu}. However,  these estimates could not help us to obtain the central limit theorems for an additive functional of two independent fBms, or more generally of two independent copies of a Gaussian process $X$ satisfying conditions {\bf (H1)-(H3)}, because the methodology developed in \cite{nx2}  only allows us to  derive obtain central limit theorems when times are fixed. 

A new technique will be introduced to extend the result in \cite{nx2} to functional central limit theorems, which is called the paring technique and  was original developed in \cite{sxy} to get limit laws for functionals of two independent fBms in the critical case $Hd=2$. Here is a rough description of this technique.  When showing the convergence of even moments, we first use the Fourier transform and arrange the ordering of the first process $X^{(1)}$ according to the ordering of its time points. Then we would see that the spatial variable $y$ with odd index $2k-1$ multiplying an  increment of $X^{(1)}$ is very close to the one with the index $2k$. The same paring also works for the second  process $X^{(2)}$. More interestingly, the paring for the first process would also match the paring for the second one.  These parings finally let us obtain the desired convergence of even moments in \cite{sxy}. In this paper, we consider the case $Hd<2$, where the intersection local exists, and the approach will be  different from the one in \cite{sxy} and new ideas and tools  will be required to derive the functional central limit theorem.

\medskip
In the following, we would illustrate how to use the paring technique in computing the limit of even moments.
\begin{proposition} \label{prop2.7} If $m$ is even, then
\[
\lim\limits_{n\to\infty}  \E[F_{n}(a_1,b_1; a_2, b_2)^{m}]=D^{m/2}_{f,H,d}\, \E\big[\Delta_{E}\Lambda\big]^{m},
\]
where $E=(a_1,b_1]\times (a_2, b_2]$.
\end{proposition}

\begin{proof}
We divide the proof into several steps. We can assume that $a_1, a _2 >0$.

\noindent \textbf{Step 1.} We show that $\E[F_{n}(a_1,b_1; a_2, b_2)^{m}]$ is asymptotically equivalent  to $I^n_{m}$ defined in (\ref{inm}), where
we impose the  upper bound  $n^{\lambda_0}\varepsilon$ for some $\lambda_0>0$ to the  even differences $\Delta u_{2i}$ and $\Delta v_{2i}$ and the lower bound $n^{1/2 } \varepsilon$  to  the odd differences $\Delta u_{2i-1}$ and $\Delta v_{2i-1}$, respectively. Note that 
\begin{align*}
\E[F_{n}(a_1,b_1; a_2, b_2)^{m}]
&=
\frac{m!}{(2\pi)^{md}}\, n^{\frac{m(Hd-2)}{2}}  \sum_{\sigma\in\mathscr{P}}  \, \int_{\R^{md}} \int_{D^m_{a_1, b_1}\times D^m_{a_2, b_2}}     \prod _{i=1}^{m} \widehat{f} (x_i)    \notag  \\
&\quad \times
      \exp \bigg( -\frac 12 {\rm Var} \Big(\sum_{i=1}^{m} x_i \cdot (X^{(1)}_{u_i}-X^{(2)}_{v_{\sigma(i)}})\Big)\bigg) \,du\, dv\,  dx.
      \end{align*}     
For $\varepsilon\in(0,1)$ and $\ell=1,2$, we let
\begin{align} \label{oml}
O_{m,\ell}
&=D^m_{a_{\ell}, b_{\ell}}\cap \big\{0<\Delta u_{2i}<n^{\lambda_0} \varepsilon, \, i=1,2,\dots, m/2 \big\} \nonumber\\
&\qquad\qquad\qquad\cap \big\{n a_{\ell} \vee n^{1/2}\varepsilon<\Delta u_1<n b_{\ell},\, n^{1/2}\varepsilon<\Delta u_{2i-1}<n(b_{\ell}-a_{\ell}),\, i=2,\dots ,m/2 \big\}, 
\end{align}
where $\lambda_0=\frac{\lambda}{(2-Hd)m}$ with the constant $\lambda$ given in Lemma \ref{chain}, and $\Delta u_{k}=u_{k}-u_{k-1}$ for $k=1,2,\dots, m$ with the convention $u_0=0$.  Since $m$ is even, $\lambda_0<1/4$.
Set
\begin{align}
I^n_{m}  \label{inm}
&=\frac{m!}{(2\pi)^{md}}\, n^{\frac{m(Hd-2)}{2}} \sum_{\sigma\in\mathscr{P}}   \int_{O_{m,1}\times O_{m,2}} \int_{\R^{md}} \prod^{m}_{i=1}\widehat{f}(x_i)  \nonumber  \\  
&\qquad\qquad\qquad \times \exp\Big(-\frac{1}{2}\Var\big(\sum\limits^{m}_{i=1} x_i\cdot (X^{(1)}_{u_i}-X^{(2)}_{v_{\sigma(i)}})\big)\Big)\, dx\, du\, dv.
\end{align}
Then
\begin{align*}  
\left|\E[F_{n}(a_1,b_1; a_2, b_2)^{m}]-I^n_{m} \right|
&\leq c_1\, n^{\frac{m(Hd-2)}{2}} \sum_{\sigma\in\mathscr{P}} \int_{D^m_{a_1, b_1}\times D^m_{a_2, b_2}-O_{m,1}\times O_{m,2}} \int_{\R^{md}} \prod^{m}_{i=1}|\widehat{f}(x_i)| \\  
&\qquad\qquad\times \exp\Big(-\frac{1}{2}\Var\big(\sum\limits^{m}_{i=1} x_i\cdot (X^{(1)}_{u_i}-X^{(2)}_{v_{\sigma(i)}})\big)\Big)\, dx\, du\, dv \\
&\leq c_2\, n^{\frac{m(Hd-2)}{2}}  \int_{D^m_{a_1, b_1}\times D^m_{a_2, b_2}-O_{m,1}\times O_{m,2}} \int_{\R^{md}} \prod^{m}_{i=1}|\widehat{f}(y_i-y_{i+1})|   \\  
&\qquad\qquad \times \exp\Big(-\frac{\kappa}{2}\sum\limits^{m}_{i=1} |y_i|^2\big[(\Delta u_i)^{2H}+(\Delta v_i)^{2H}\big]\Big)\, dy\, du\, dv,  
\end{align*}
where  in the last inequality we applied Cauchy-Schwartz inequality and then used the symmetry of the product of the Fourier transforms.

Applying the estimate (\ref{eku2}) in Lemma \ref{chain} to the right hand side of the above inequality,  we obtain
\begin{align} \label{p9}
&\limsup\limits_{n\to\infty} \Big|\E[F_{n}(a_1,b_1; a_2, b_2)^{m}]-I^n_{m} \Big| \nn \\
&\leq c_3 \limsup\limits_{n\to\infty}  n^{\frac{m(Hd-2)}{2}}  \int_{D^m_{a_1, b_1}\times D^m_{a_2, b_2}-O_{m,1}\times O_{m,2}} \int_{\R^{md}} \prod^{m/2}_{k=1}|\widehat{f}(y_{2k})|^2  \nonumber  \\  
&\qquad\qquad\qquad \times \exp\Big(-\frac{\kappa}{2}\sum\limits^{m}_{i=1} |y_i|^2\big[(\Delta u_i)^{2H}+(\Delta v_i)^{2H}\big]\Big)\, dy\, du\, dv.
\end{align}
Recall the definition of $O_{m,\ell}$ in (\ref{oml}). We see that for all $(u,v)\in D^m_{a_1, b_1}\times D^m_{a_2, b_2}-O_{m,1}\times O_{m,2}$ there exist 1)
some $\ell\in\{1,\dots, m/2\}$ such that 
 $\Delta u_{2\ell}\geq n^{\lambda_0}\varepsilon$ or $\Delta v_{2\ell}\geq n^{\lambda_0}\varepsilon$; or 2)  some $\ell\in\{2,\dots, m/2\}$ such that  $\Delta u_{2\ell-1}\leq n^{1/2} \varepsilon$ or $\Delta v_{2\ell-1}\leq n^{1/2} \varepsilon$. 
 Therefore, the right hand side of (\ref{p9}) is less than a constant multiple of 
  \begin{align*}
& \limsup\limits_{n\to\infty}  n^{\frac{m(Hd-2)}{2}} \sum^{m/2}_{\ell=1} \int_{D^m_{a, b}\times D^m_{a, b}} \int_{\R^{md}} \prod^{m/2}_{k=1}|\widehat{f}(y_{2k})|^2 {\bf 1}_{\{\Delta u_{2\ell}\geq n^{\lambda_0} \varepsilon \} \cup \{ \Delta v_{2\ell}\geq n^{\lambda_0} \varepsilon\}} \nonumber  \\  
&\qquad\qquad\qquad \times \exp\Big(-\frac{\kappa}{2}\sum\limits^{m}_{i=1} |y_i|^2\big[(\Delta u_i)^{2H}+(\Delta v_i)^{2H}\big]\Big)\, dy\, du\, dv\\ 
&\qquad+\limsup\limits_{n\to\infty}  n^{\frac{m(Hd-2)}{2}} \sum^{m/2}_{\ell=2} \int_{D^m_{a, b}\times D^m_{a, b}} \int_{\R^{md}} \prod^{m/2}_{k=1}|\widehat{f}(y_{2k})|^2 {\bf 1}_{\{\Delta u_{2\ell-1}\leq n^{1/2} \varepsilon \} \cup \{ \Delta v_{2\ell-1}\leq n^{1/2} \varepsilon\}} \nonumber  \\  
&\qquad\qquad\qquad \times \exp\Big(-\frac{\kappa}{2}\sum\limits^{m}_{i=1} |y_i|^2\big[(\Delta u_i)^{2H}+(\Delta v_i)^{2H}\big]\Big)\, dy\, du\, dv \\
& \qquad = A^{(1)} + A^{(2)},
\end{align*}
where $a=\min\{a_1,a_2\}$ and $b=\max\{b_1,b_2\}$.

Using Cauchy-Schwarz inequality as in Subsection 3.1 and then doing some calculation yield
\begin{align*}
&A^{(1)} \leq c_4 \limsup\limits_{n\to\infty} \sum^{m/2}_{\ell=1} \int_{  [n^{\lambda_0} \varepsilon, nb] \times [0,nb] \cup
  [0,nb] \times  [n^{\lambda_0} \varepsilon, nb] }  \int_{\R^{d}} |\widehat{f}(y_{2\ell})|^2  \\
   & \qquad \qquad  \qquad \qquad  \times \exp\Big(-\frac{\kappa}{2}|y_{2\ell}|^2\big(s_{2\ell}^{2H}+ r_{2\ell}^{2H}\big)\Big)\, dy_{2\ell}\, ds_{2\ell}dr_{2\ell}=0,
  \end{align*}
  and
  \begin{align*}
  A^{(2)} \le c_5 \limsup\limits_{n\to\infty}  n^{Hd-2} \sum^{m/2}_{\ell=2} \int _{ [0, n^{1/2} \varepsilon ]\times [0, nb ] \cup [[0, nb] \times [0, n^{1/2} \varepsilon]}
   \big(s_{2\ell-1}^{2H}+r_{2\ell-1}^{2H}\big)^{-\frac{d}{2}} ds_{2\ell-1}\, dr_{2\ell-1}=0.
\end{align*}
  So $\limsup\limits_{n\to\infty} \left|\E[F_{n}(a_1,b_1; a_2, b_2)^{m}]-I^n_{m} \right|=0$.

\noindent \textbf{Step 2.} We next show that $I^n_{m}$  is asymptotically equal to $\overline{I}^{n}_{m}$ defined in (\ref{gamma1}) below.  Making the change of variables $y_i=\sum\limits^{m}_{j=i}x_j$ for $i=1,2,\dots, m$, we can write
\begin{align*}
I^n_{m}
&=\frac{m!}{(2\pi)^{md}}\, n^{\frac{m(Hd-2)}{2}} \sum_{\sigma\in\mathscr{P}}  \int_{O_{m,1}\times O_{m,2}} \int_{\R^{md}} \prod^{m}_{i=1}\widehat{f}(y_i-y_{i+1}) \exp\Big(-\frac{1}{2}\Var\big(\sum\limits^{m}_{i=1} y_i\cdot  \Delta X_{u_i}\big)\Big)\\
&\qquad\qquad\qquad \times \exp\Big(-\frac{1}{2}\Var\big(\sum\limits^{m}_{i=1} \sum^{m}_{j=i} ( y_{\sigma(j)}-y_{\sigma(j)+1} )\cdot \Delta X_{v_i}\big)\Big)\, dy\, du\, dv.
\end{align*}

For any $\varepsilon\in(0,1)$, define 
\begin{align} \label{gamma1}
\overline{I}^{n}_{m}
&=\frac{m!}{(2\pi)^{md}}\, n^{\frac{m(Hd-2)}{2}} \sum_{\sigma \in \mathscr{P}} \int_{O_{m,1}\times O_{m,2}} \int_{T^{\sigma}_{\varepsilon}} \prod^{m}_{i=1}\widehat{f}(y_i-y_{i+1}) \exp\Big(-\frac{1}{2}\Var\big(\sum\limits^{m}_{i=1} y_i\cdot \Delta X_{u_i} \big)\Big) \nonumber \\
&\qquad\qquad\qquad \times \exp\Big(-\frac{1}{2}\Var\big(\sum\limits^{m}_{i=1} \sum^{m}_{j=i} ( y_{\sigma(j)}-y_{\sigma(j)+1} )\cdot  \Delta X_{v_i}\big)\Big)\, dy\, du\, dv,
\end{align}
where 
\begin{align*}
T^{\sigma}_{\varepsilon}=\R^{md}\cap \Big\{|y_{2k-1}|<\varepsilon,   \Big|\sum^{m}_{j={2k-1}} ( y_{\sigma(j)}-y_{\sigma(j)+1})\Big|<\varepsilon,\; k=1,2,\dots, m/2\Big\}.
\end{align*}

Recall the change of variables $y_i=\sum\limits^{m}_{j=i}x_j$ for $i=1,2,\dots, m$. Then $\overline{I}^{n}_{m}$ in (\ref{gamma1}) can also be written as
\begin{align*} 
\overline{I}^{n}_{m}
&=\frac{m!}{(2\pi)^{md}}\, n^{\frac{m(Hd-2)}{2}}\sum_{\sigma\in\mathscr{P}}\int_{O_{m,1} \times O_{m,2}} \int_{\overline{T}^{\sigma}_{\varepsilon}}  \prod^{m}_{i=1}\widehat{f}(x_i) \exp\Big(-\frac{1}{2}\Var\big(\sum\limits^{m}_{i=1} \sum\limits^{m}_{j=i}x_j\cdot \Delta X_{u_i})\big)\Big) \nonumber \\
&\qquad\qquad\qquad \times \exp\Big(-\frac{1}{2}\Var\big(\sum\limits^{m}_{i=1} \sum^{m}_{j=i} x_{\sigma(j)} \cdot  \Delta X_{v_i} \big)\Big)\, dx\, du\, dv,
\end{align*}
where 
\begin{align} \label{domain}
\overline{T}^{\sigma}_{\varepsilon}=\R^{md}\cap \Big\{\Big|\sum\limits^{m}_{j=2k-1}x_j\Big|<\varepsilon,\;  \Big|\sum^{m}_{j={2k-1}} x_{\sigma(j)}\Big|<\varepsilon,\; k=1,2,\dots, m/2\Big\}.
\end{align}

Using Cauchy-Schwarz inequality as in Subsection 3.1 and then making the change of variables $y_i=\sum\limits^{m}_{j=i}x_j$ for $i=1,2,\dots, m$, we obtain 
\begin{align*}
|I^n_{m}-\overline{I}^{n}_{m}| 
&\leq c_6\, n^{\frac{m(Hd-2)}{2}} \int_{O_{m,1} \times O_{m,2}} \int_{\R^{md}-T_{\varepsilon}} \prod^{m}_{i=1}|\widehat{f}(y_i-y_{i+1})|    \exp\Big(-\frac{1}{2}\Var\big(\sum\limits^{m}_{i=1} y_i\cdot \Delta X_{u_i} \big)\Big) \nonumber \\
&\qquad\qquad\qquad \times \exp\Big(-\frac{1}{2}\Var\big(\sum\limits^{m}_{i=1} y_i \cdot  \Delta X_{v_i}\big)\Big)\, dy\, du\, dv,
\end{align*}
where 
\begin{align*}
T_{\varepsilon}=\R^{md}\cap \big\{ |y_{2k-1}|<\varepsilon,\; k=1,2,\dots, m/2\big\}.
\end{align*}

Now, by Lemma \ref{chain} and the  nondeterminism property (\ref{eku1}),  we can write 
\begin{align*}
|I^n_{m}-\overline{I}^{n}_{m}|
&\leq c_{7}n^{-\lambda}+c_{7}\, n^{\frac{m(Hd-2)}{2}}\int_{O_{m,1} \times O_{m,2}}\int_{\R^{md}-T_{\varepsilon}} \prod^{m/2}_{j=1}|\widehat{f}(y_{2j})|^2\\
&\qquad\qquad\qquad\qquad \times \exp\Big(-\frac{\kappa}{2}\sum\limits^{m}_{i=1} |y_i|^2\big[(\Delta u_i)^{2H}+(\Delta v_i)^{2H}\big]\Big)\, dy \, du\, dv\\
&\leq c_{7}n^{-\lambda}+c_{8}\, n^{Hd-2} \int^{nb}_{n^{1/2}\varepsilon}\int^{nb}_{n^{1/2}\varepsilon}\int_{|x|\geq \varepsilon} \exp\Big(-\frac{\kappa}{2}|x|^2(s^{2H}+t^{2H})\Big) \, dx\, ds\, dt\\
&\leq c_{7}n^{-\lambda}+c_{9}\, n^{Hd-2} e^{-\frac{\kappa}{2}\varepsilon^{2+2H} n^{H}} \int^{nb}_{n^{1/2}\varepsilon}\int^{nb}_{n^{1/2}\varepsilon}\int_{|x|\geq \varepsilon} \exp\Big(-\frac{\kappa}{4}|x|^2(s^{2H}+t^{2H})\Big) \, dx\, ds\, dt\\
&\leq c_{7}n^{-\lambda}+c_{10}\, e^{-\frac{\kappa}{2}\varepsilon^{2+2H} n^{H}},
\end{align*}
where $b=\max\{b_1,b_2\}$. This implies that $\limsup\limits_{n\to\infty} |I^n_{m}-\overline{I}^{n}_{m}|=0$.
 
\noindent \textbf{Step 3.} For any $\sigma\in\mathscr{P}$, let 
\begin{align*} 
\overline{I}^{n,\sigma}_{m}
&=\frac{m!}{(2\pi)^{md}}\, n^{\frac{m(Hd-2)}{2}} \int_{O_{m,1} \times O_{m,2}} \int_{\overline{T}^{\sigma}_{\varepsilon}}  \prod^{m}_{i=1}\widehat{f}(x_i) \exp\Big(-\frac{1}{2}\Var\big(\sum\limits^{m}_{i=1} \sum\limits^{m}_{j=i}x_j \cdot \Delta X_{u_i} \big)\Big) \nonumber \\
&\qquad\qquad\qquad \times \exp\Big(-\frac{1}{2}\Var\big(\sum\limits^{m}_{i=1} \sum^{m}_{j=i} x_{\sigma(j)} \cdot  \Delta X_{v_i} \big)\Big)\, dx\, ds.
\end{align*}
Then $\overline{I}^{n}_{m}= \sum\limits_{\sigma\in\mathscr{P}} \overline{I}^{n,\sigma}_{m}$ and in the sequel we will study the asymptotic behavior of $\overline{I}^{n,\sigma}_{m}$ for a fixed $\sigma$.  To do this we consider a partition of the set of permutations 
\[
\mathscr{P}= \mathscr{P}_0 \cup \mathscr{P}_1,
\]
where $\mathscr{P}_1$ is the set of permutations $\sigma \in \mathscr{P}$ such that  the collection of pairs $\big\{\{2k,2k-1\},\, k=1,2,\dots, m/2\big\}$ is invariant by $\sigma$,  in the sense that
\[
\big\{\{2k,2k-1\},\, k=1,2,\dots, m/2\big\}=\big\{\{\sigma(2k),\sigma(2k-1)\},\, k=1,2,\dots, m/2\big\}.
\]
 
\noindent \textbf{Step 4.}  We first  study $ \overline{I}^{n,\sigma}_{m} $ for $\sigma \in \mathscr{P}_0$.
For any $\sigma\in \mathscr{P}_0$, there exist $j, k,\ell\in\{1,2,\dots, m/2\}$ with $k\neq \ell$ such that 
\begin{align} \label{permutation}
\sigma(2j)\in\{2k,2k-1\}\; \text{and}\; \sigma(2j-1)\in\{2\ell,2\ell-1\}.
\end{align}
Recall the definition of $\overline{T}^{\sigma}_{\varepsilon}$ in (\ref{domain}). For any $k=1,2,\dots, m/2$,
\begin{align} \label{paring}
|x_{2k}+x_{2k-1}|\leq (m/2-k+1)\varepsilon\quad \text{and}\quad |x_{\sigma(2k)}+x_{\sigma(2k-1)}|\leq (m/2-k+1)\varepsilon.
\end{align}
We claim that 
\begin{align*} 
|x_{2k}-x_{2\ell}|\leq 2m\varepsilon \quad \text{or}\quad |x_{2k}+x_{2\ell}|\leq 2m\varepsilon.
\end{align*}
In fact, from (\ref{permutation}), there are only four possibilities for the values of $\sigma(2j)$ and $\sigma(2j-1)$: (1) $\sigma(2j)=2k$ and $\sigma(2j-1)=2\ell$; (2) $\sigma(2j)=2k$ and $\sigma(2j-1)=2\ell-1$; (3) $\sigma(2j)=2k-1$ and $\sigma(2j-1)=2\ell$;  (4) $\sigma(2j)=2k-1$ and $\sigma(2j-1)=2\ell-1$. In the first case, the claim follows from (\ref{paring}) directly. In the second and third cases, 
\begin{align*} 
|x_{2k}-x_{2\ell}|\leq |x_{2k}-(-1)^{\sigma(2j)}x_{\sigma(2j)}|+|x_{\sigma(2j)}+x_{\sigma(2j-1)}|+|(-1)^{\sigma(2j-1)}x_{\sigma(2j-1)}-x_{2\ell}|\leq 2m\varepsilon.
\end{align*}
In the last case,
\begin{align*} 
|x_{2k}+x_{2\ell}|\leq |x_{2k}+x_{\sigma(2j)}|+|x_{\sigma(2j)}+x_{\sigma(2j-1)}|+|x_{\sigma(2j-1)}+x_{2\ell}|\leq 2m\varepsilon.
\end{align*}

We next show that 
\[
|y_{2k}-y_{2\ell}|\leq 4m\varepsilon \quad \text{or}\quad |y_{2k}-y_{2\ell}|\leq 4m\varepsilon.
\] 
Without loss of generality, we can assume that $k<\ell$. Then
\begin{align*}
|y_{2k}-y_{2\ell}|
&=\bigg|\sum^{2\ell}_{j=2k+1} x_j+x_{2k}-x_{2\ell}\bigg|\leq 4m\varepsilon
\end{align*}
if $|x_{2k}-x_{2\ell}|\leq 2m\varepsilon$, and 
\begin{align*}
|y_{2k}+y_{2\ell}|
&=|2\sum^{m}_{j=2\ell+1} x_j+\sum^{2\ell}_{j=2k+1} x_j+x_{2k}+x_{2\ell}|\leq 4m\varepsilon
\end{align*}
if $|x_{2k}+x_{2\ell}|\leq 2m\varepsilon$.

Using  arguments  similar as  those in Subsection 3.1 and then Lemma \ref{chain}, we can get
\begin{align*} 
&|\overline{I}^{n,\sigma}_{m}|\\
&\leq c_{11}\, n^{\frac{m(Hd-2)}{2}}  \int_{O_{m,1} \times O_{m,2}} \int_{\overline{T}^{\sigma}_{\varepsilon}} \prod^{m}_{i=1}|\widehat{f}(x_i)| \exp\Big(-\frac{1}{2}\Var\big(\sum\limits^{m}_{i=1} \sum\limits^{m}_{j=i}x_j \cdot \Delta X_{u_i} \big)\Big) \nonumber \\
&\qquad\qquad\qquad \times \exp\Big(-\frac{1}{2}\Var\big(\sum\limits^{m}_{i=1} \sum^{m}_{j=i} x_{j} \cdot  \Delta X_{v_i}\big)\Big)\, dx\, du\, dv\\
&\leq c_{12}n^{-\lambda}+c_{12}\,\sum_{1\leq k\neq \ell\leq m/2}\, n^{\frac{m(Hd-2)}{2}} \int_{O_{m,1} \times O_{m,2}} \int_{\R^{md}} \prod^{m/2}_{j=1}|\widehat{f}(y_{2j})|^2\\
&\qquad\qquad\qquad \times \exp\Big(-\frac{\kappa_{H}}{2}\sum\limits^{m}_{i=1} |y_i|^2 [(\Delta u_i)^{2H}+(\Delta v_i)^{2H}]\Big) \, {\bf 1}_{\{|y_{2k}\pm y_{2\ell}|\leq 4m\varepsilon\}}\, dy\, du\, dv,
\end{align*}
where 
\[
\{|y_{2k}\pm y_{2\ell}|\leq 4m\varepsilon\}=\{|y_{2k}+y_{2\ell}|\leq4m\varepsilon\}\cup\{|y_{2k}-y_{2\ell}|\leq 4m\varepsilon\}.
\]

Integrating with respect to all $u$, $v$, and $y_i$ with $i\neq 2k,2\ell$ gives
\begin{align*}
\limsup\limits_{n\to\infty}|\overline{I}^{n,\sigma}_{m}|\leq c_{13}\, \int_{\R^{2d}} |\widehat{f}(x)|^2 |\widehat{f}(y)|^2|x|^{-\frac{2}{H}}|y|^{-\frac{2}{H}}\,  {\bf 1}_{\{|x-y|\leq 4m\varepsilon\}}\, dx\, dy
\end{align*}
for all $\sigma\in\mathscr{P}_0$. Taking into account that $\varepsilon$ is arbitrary, we see that there  will be no contribution in the limit for $\sigma\in \mathscr{P}_0$. 

 \noindent \textbf{Step 5.}  Now we analyze  $ \overline{I}^{n,\sigma}_{m} $ for $\sigma \in \mathscr{P}_1$.
 Note that $\overline{I}^{n,\sigma}_{m}$ can be written as
\begin{align*}
\overline{I}^{n,\sigma}_{m}
&=\frac{m!}{(2\pi)^{md}}\, n^{\frac{m(Hd-2)}{2}} \int_{O_{m,1}\times O_{m,2}} \int_{T^{\sigma}_{\varepsilon}} \prod^{m}_{i=1}\widehat{f}(y_i-y_{i+1}) \exp\Big(-\frac{1}{2}\Var\big(\sum\limits^{m}_{i=1} y_i\cdot \Delta X_{u_i} \big)\Big) \nonumber \\
&\qquad\qquad\qquad \times \exp\Big(-\frac{1}{2}\Var\big(\sum\limits^{m}_{i=1} \sum^{m}_{j=i} ( y_{\sigma(j)}-y_{\sigma(j)+1} )\cdot  \Delta X_{v_i} \big)\Big)\, dy\, du\, dv.
\end{align*}
Fix $\gamma>1$. Let $T^{\sigma}_{\varepsilon,1}=T^{\sigma}_{\varepsilon}-T^{\sigma}_{\varepsilon,2}$ where 
\begin{align} \label{tsv2}
T^{\sigma}_{\varepsilon,2}=T^{\sigma}_{\varepsilon}\cap \{|y_{2i}|>\gamma \varepsilon:\, i=1,2,\dots, m/2\}
\end{align}
and define
\begin{align*}
\overline{I}^{n,\sigma}_{m,\gamma, 1}
&=\frac{m!}{(2\pi)^{md}}\, n^{\frac{m(Hd-2)}{2}} \int_{O_{m,1}\times O_{m,2}} \int_{T^{\sigma}_{\varepsilon,1}} \prod^{m}_{i=1}\widehat{f}(y_i-y_{i+1}) \exp\Big(-\frac{1}{2}\Var\big(\sum\limits^{m}_{i=1} y_i\cdot \Delta X_{u_i} \big)\Big) \nonumber \\
&\qquad\qquad\qquad \times \exp\Big(-\frac{1}{2}\Var\big(\sum\limits^{m}_{i=1} \sum^{m}_{j=i} ( y_{\sigma(j)}-y_{\sigma(j)+1} )\cdot \Delta X_{v_i} \big)\Big)\, dy\, du\, dv,
\end{align*}
\begin{align*}
\overline{I}^{n,\sigma}_{m,\gamma, 2}
&=\frac{m!}{(2\pi)^{md}}\, n^{\frac{m(Hd-2)}{2}} \int_{O_{m,1}\times O_{m,2}} \int_{T^{\sigma}_{\varepsilon,2}} \prod^{m}_{i=1}\widehat{f}(y_i-y_{i+1}) \exp\Big(-\frac{1}{2}\Var\big(\sum\limits^{m}_{i=1} y_i\cdot \Delta X_{u_i}\big)\Big) \nonumber \\
&\qquad\qquad\qquad \times \exp\Big(-\frac{1}{2}\Var\big(\sum\limits^{m}_{i=1} \sum^{m}_{j=i} ( y_{\sigma(j)}-y_{\sigma(j)+1} )\cdot \Delta X_{v_i} \big)\Big)\, dy\, du\, dv,
\end{align*}
and 
\begin{align*}
\overline{I}^{n,\sigma}_{m,\gamma, 3}
&=\frac{m!}{(2\pi)^{md}}\, n^{\frac{m(Hd-2)}{2}} \int_{O_{m,1}\times O_{m,2}} \int_{T^{\sigma}_{\varepsilon,2}} \prod^{m/2}_{j=1} |\widehat{f}(y_{2j})|^2 \exp\Big(-\frac{1}{2}\Var\big(\sum\limits^{m}_{i=1} y_i\cdot \Delta X_{u_i} \big)\Big) \nonumber \\
&\qquad\qquad\qquad \times \exp\Big(-\frac{1}{2}\Var\big(\sum\limits^{m}_{i=1} \sum^{m}_{j=i} ( y_{\sigma(j)}-y_{\sigma(j)+1} )\cdot \Delta X_{v_i} \big)\Big)\, dy\, du\, dv.
\end{align*}
Obviously, $\overline{I}^{n,\sigma}_{m}=\overline{I}^{n,\sigma}_{m,\gamma, 1}+\overline{I}^{n,\sigma}_{m,\gamma, 2}$.  We will show the following two properties:

\smallskip
\noindent
{\it (i)}   For any $\sigma\in\mathscr{P}_1$ and for some constant $c_{14}>0$,
\begin{equation} \label{eku6}
\limsup\limits_{n\to\infty} |\overline{I}^{n,\sigma}_{m,\gamma, 1} |  \le c_{14} \int_{|y|\leq \gamma \varepsilon} |\widehat{f}(y)|^2 |y|^{-\frac{2}{H}}\, dy.
\end{equation}

\smallskip
\noindent
{\it (ii)}   For any $\sigma\in\mathscr{P}_1$,
 \begin{equation} \label{eku7}
 \limsup\limits_{n\to\infty} |\overline{I}^{n,\sigma}_{m,\gamma, 2}-\overline{I}^{n,\sigma}_{m,\gamma, 3}|=0.
 \end{equation}

\smallskip
\noindent
{\it Proof of (\ref{eku6})}: This follows from Lemma \ref{chain} and arguments similar to those in Subsection 3.1.

\smallskip
\noindent
{\it Proof of (\ref{eku7})}: 
Using Cauchy-Schwarz inequality and the boundedness of $\widehat{f}$,
\begin{align*}
&\int_{O_{m,1}\times O_{m,2}} \int_{T^{\sigma}_{\varepsilon,2}} \bigg|\prod^{m}_{i=1}\widehat{f}(y_i-y_{i+1})-\prod^{m/2}_{j=1}|\widehat{f}(y_{2j})|^2\bigg| \exp\Big(-\frac{1}{2}\Var\big(\sum\limits^{m}_{i=1} y_i\cdot  \Delta X_{u_i}\big)\Big) \nonumber \\
&\qquad\qquad\qquad \times \exp\Big(-\frac{1}{2}\Var\big(\sum\limits^{m}_{i=1} \sum^{m}_{j=i} ( y_{\sigma(j)}-y_{\sigma(j)+1} )\cdot  \Delta X_{v_i} \big)\Big)\, dy\, du\, dv\\
&\leq  c_{15} \Bigg[\int_{O_{m,1}\times O_{m,1}} \int_{T^{\sigma}_{\varepsilon,2}} \bigg|\prod^{m}_{i=1}\widehat{f}(y_i-y_{i+1})-\prod^{m/2}_{j=1}|\widehat{f}(y_{2j})|^2 \bigg| \exp\Big(-\frac{1}{2}\Var\big(\sum\limits^{m}_{i=1} y_i\cdot \Delta X_{u_i}\big)\Big) \nonumber \\
&\qquad\qquad\qquad \times \exp\Big(-\frac{1}{2}\Var\big(\sum\limits^{m}_{i=1}  y_i\cdot  \Delta X_{v_i}\big)\Big)\, dy\, du\, dv \Bigg]^{1/2} \\
&\qquad\qquad \times \Bigg[\int_{O_{m,2}\times O_{m,2}} \int_{T^{\sigma}_{\varepsilon,2}}\exp\Big(-\frac{1}{2}\Var\big(\sum\limits^{m}_{i=1}\sum^{m}_{j=i} ( y_{\sigma(j)}-y_{\sigma(j)+1} )\cdot \Delta X_{u_i}\big)\Big) \nonumber \\
&\qquad\qquad\qquad\qquad \times \exp\Big(-\frac{1}{2}\Var\big(\sum\limits^{m}_{i=1} \sum^{m}_{j=i} ( y_{\sigma(j)}-y_{\sigma(j)+1} )\cdot \Delta X_{v_i}\big)\Big)\, dy\, du\, dv \Bigg]^{1/2} \\
&  =: c_{15} ( C^{(1)}_nC^{(2)}_n)^{1/2}.
\end{align*}

By Lemma \ref{chain},
\begin{equation}
\label{eku8}
n^{\frac{m(Hd-2)}{2}} C^{(1)}_n \leq c_{16} n^{-\lambda}.
\end{equation}

On the other hand, using  the nondeterminism property (\ref{eku1}) and integrating with respect to $y$ gives
\begin{align*}
 n^{\frac{m(Hd-2)}{2} }C^{(2)}_n 
&\leq c_{17} n^{\frac{m(Hd-2)}{2}}\int_{O_{m,2}\times O_{m,2}} \prod\limits^{m}_{i=1} \big[(u_i-u_{i-1})^{2H}+(v_i-v_{i-1})^{2H}\big]^{-\frac{d}{2}}\, du\, dv\\
&\leq c_{18} n^{\frac{m\lambda_0(2-Hd)}{2}}\\
&=c_{18} n^{\frac{\lambda}{2}},
\end{align*}
where in the second inequality we use the definition of $O_{m,2}$ in (\ref{oml}).

Therefore, $\limsup_{n\to\infty} |\overline{I}^{n,\sigma}_{m,\gamma, 2}-\overline{I}^{n,\sigma}_{m,\gamma, 3}|\leq \limsup\limits_{n\to\infty} c_{19}\, n^{-\frac{\lambda}{4}}=0$.

\noindent \textbf{Step 6.} We will finally derive  the limit of $I^n_m$ as $n\to\infty$. Consider the decomposition
 \begin{align*}
 \Var\big(\sum_{i=1} ^m y_i \cdot \Delta X_{u_i} \big)
&=  \Var\big(\sum_{i=1, i \, odd } ^m y_i \cdot \Delta X_{u_i} \big)+
\sum_{i=1, i \, even } ^m \Var\big(y_i \cdot \Delta X_{u_i} \big) \\
&\qquad\qquad\qquad+ \Cov  \Big( \sum_{i=1,  i \,  odd } ^m y_i \cdot \Delta X_{u_i}, \sum_{i=1, i \, even } ^m y_i \cdot \Delta X_{u_i} \Big)\\
&\qquad\qquad\qquad\qquad+ \sum^m_{i,j=2,\, even, i\neq j}\Cov  \Big( y_i \cdot \Delta X_{u_i}, y_j\cdot \Delta X_{u_j} \Big).
\end{align*}
By the definition of $O_{m,\ell}$ in (\ref{oml}),  if $i$ is odd and $j$ is even, then
if $i>j$ we have $\frac {\Delta u_i}{\Delta u_j} \ge n^{\frac 12-\lambda_0}$ and if $i<j$, then  $\frac {\Delta u_j}{\Delta u_i} \le n^{ \lambda_0-\frac 12}$. 
Therefore, by Hypothesis {\bf (H3)},
we can write
 \begin{align*}
& \left| \Cov  \Big( \sum_{i=1,  i \,  odd } ^m y_i \cdot \Delta X_{u_i}, \sum_{i=1, i \, even } ^m y_i \cdot \Delta X_{u_i} \Big) \right| \\
& \qquad \qquad  \le \sum_{i=1,  i \,  odd } ^m \sum_{j=1, j \, even } ^m |y_i|   |y_j| \beta(n^{\frac 12 -\lambda_0})\sqrt{ \E (\Delta X_{u_i})^2  \E (\Delta X_{u_j})^2 }\\
& \qquad \qquad \le \frac m4  \beta(n^{\frac 12 -\lambda_0}) \left(     \sum_{i=1,  i \,  odd }^m |y_i |^2  \E (\Delta X_{u_i})^2 +  \sum_{j=1,  j\, even}^m |y_i |^2  \E (\Delta X_{u_j})^2  \right).
 \end{align*}
On the other hand, by the definition of $O_{m,\ell}$ in (\ref{oml}), if $i<j$ are even, then
$\frac {\Delta u_i} {\Delta u_{j-1}} \le n^{\lambda_0-\frac 12}$ and  $\frac {\Delta u_j} {\Delta u_{j-1}} \le n^{\lambda_0-\frac 12}$. Therefore,  by Hypothesis {\bf (H3)}  we can write
  \begin{align*}
&\left| \sum^m_{i,j=2,\, even, i\neq j}\Cov  \Big( y_i \cdot \Delta X_{u_i}, y_j\cdot \Delta X_{u_j} \Big)\right|\\
&\qquad \qquad\leq \sum^m_{i,j=2,\, even, i\neq j} |y_i|   |y_j| \beta(n^{\frac 12 -\lambda_0}) \sqrt{ \E (\Delta X_{u_i})^2  \E (\Delta X_{u_j})^2 }\\
&\qquad \qquad \leq \frac m2 \beta(n^{\frac 12 -\lambda_0})\sum^m_{i=2,\, even} |y_i |^2  \E (\Delta X_{u_i})^2.
\end{align*}

By hypothesis {\bf (H2)}, there exists $\gamma_0\geq 1$ such that $\E (\Delta X_{u_i})^2\leq (\alpha_2 + \phi_2(\frac{1}{\gamma_0})) |\Delta u_i|^{2H}$ if $\frac{\Delta u_i}{u_{i-1}}\leq 1/\gamma_0$. Moreover, if $\frac{\Delta u_i}{u_{i-1}}\geq 1/\gamma_0$, by self-similarity, then 
\begin{align*}
\E (\Delta X_{u_i})^2\leq 2\E (X_{u_i})^2+2\E (X_{u_{i-1}})^2\leq 4|u_i|^{2H}\E (X_1)^2\leq 4\,\E (X_1)^2 (\gamma_0+1)^{2H}|\Delta u_i|^{2H}.
\end{align*}
Now by Hypothesis {\bf (H1)},
 \[
  \sum_{i=1,  i \,  odd }^m |y_i |^2  \E (\Delta X_{u_i})^2 \le  \frac 1{\kappa}  \Big[\alpha_2 + \phi_2(\frac{1}{\gamma_0})+4\,\E (X_1)^2 (\gamma_0+1)^{2H}\Big] \Var    \left(  \sum_{i=1,  i \,  odd }^m y_i \cdot \Delta X_{u_i} \right).
  \]
Using Hypothesis {\bf (H2)} and the definition of $O_{m,\ell}$ in (\ref{oml}),  $\Var\big(\sum\limits_{i=1} ^m y_i \cdot \Delta X_{u_i} \big)$ is between 
\[
(1-\beta(n))\Var\big(\sum_{i=1, i \,  odd } ^m y_i \cdot \Delta X_{u_i} \big)  +(1-\beta(n))(\alpha_2-\phi_1(n^{\lambda_0-\frac{1}{2}}))\sum_{i=1,  i\, even } ^m |y_i|^2 |\Delta u_i| ^{2H},
\]
and 
\[
(1+\beta(n))\Var\big(\sum_{i=1, i \,  odd } ^m y_i \cdot \Delta X_{u_i} \big)  +(1+\beta(n))(\alpha_2+\phi_2(n^{\lambda_0-\frac{1}{2}}))\sum_{i=1,  i\, even } ^m |y_i|^2 |\Delta u_i| ^{2H},
\]
where $\beta(n)=c_{20} \beta(n^{\frac 12 -\lambda_0})$.

Similarly, on $O_{m,\ell}$, $\Var\big(\sum\limits_{i=1, i \,  odd } ^m y_i \cdot (X_{u_i}-X_{u_{i-2}}) \big)$ is between 
\[
(1-\overline{\beta}(n))\Var\big(\sum\limits_{i=1, i \,  odd } ^m y_i \cdot \Delta X_{u_i} \big)
\]
and
\[
(1+\overline{\beta}(n))\Var\big(\sum\limits_{i=1, i \,  odd } ^m y_i \cdot \Delta X_{u_i} \big).
\]
where $\overline{\beta}(n)=c_{21} \beta(n^{\frac 12 -\lambda_0})+c_{22}\frac{n^{2H(\lambda_0-\frac 12)}}{\kappa}$.

Then $\limsup\limits_{n\to\infty} \overline{I}^{n,\sigma}_{m,\gamma, 3}$ is less than
\begin{align} \label{sum}
&\limsup\limits_{n\to\infty} \frac{m!}{(2\pi)^{md}}\, n^{\frac{m(Hd-2)}{2}} \int_{O_{m,1} \times O_{m,2}} \int_{T^{\sigma}_{\varepsilon,2}}  \prod^{m/2}_{j=1}|\widehat{f}(y_{2j})|^2 \nonumber\\
& \qquad\times \exp\Big(-\frac{1}{2}\frac{1-\beta(n)}{1+\overline{\beta}(n)}\Var\big(\sum\limits^{m}_{i=1, i\, odd} y_i\cdot (X_{u_i}-X_{u_{i-2}})\big)\Big)\nonumber\\
&\qquad\times\exp\Big(-\frac{1}{2}(1-\beta(n))(\alpha_2-\phi_1(n^{\lambda_0-\frac{1}{2}}))\sum\limits^{m}_{i=1, i\, even}|y_i|^2(\Delta u_i)^{2H}\Big) \nonumber\\
&\qquad\times \exp\Big(-\frac{1}{2}\frac{1-\beta(n)}{1+\overline{\beta}(n)}\Var\big(\sum\limits^{m}_{i=1, i\, odd} \sum^{m}_{j=i} ( y_{\sigma(j)}-y_{\sigma(j)+1} )\cdot (X_{v_i}-X_{v_{i-2}})\big)\Big)\nonumber\\
&\qquad \times \exp\Big(-\frac{1}{2}(1-\beta(n))(\alpha_2-\phi_1(n^{\lambda_0-\frac{1}{2}}))\sum\limits^{m}_{i=1, i\, even} |\sum^{m}_{j=i} ( y_{\sigma(j)}-y_{\sigma(j)+1} )|^2(\Delta v_i)^{2H}\Big)\, dy\, du\, dv.
\end{align}

Recall the definition of $\mathscr{P}_1$ in \textbf{Step 3}. It is easy to see that $\# \mathscr{P}_1=2^{\frac{m}{2}}(\frac{m}{2})!$. Moreover, for any $\sigma\in \mathscr{P}_1$,  the expression of summation
\begin{align*} \label{sum}
\sum^{m}_{j=i} ( y_{\sigma(j)}-y_{\sigma(j)+1})
\end{align*}
on the right-hand side of (\ref{sum})  after simplification only has two possibilities. One is that it consists of only variables $y$ with odd indices when $i$ is odd. The other is that there is only one variable $y$ with even index in its expression when $i$ is even. Note that all variables $y$ with odd indices are in the ball centered at the origin with radius $\varepsilon$ and $\varepsilon$ is a positive constant which could be arbitrary small.  Recall the definition of $T^{\sigma}_{\varepsilon,2}$ in (\ref{tsv2}). Choosing $\gamma$ large enough gives 
\[
\sum\limits^{m}_{i=1, i\, even} |\sum^{m}_{j=i} ( y_{\sigma(j)}-y_{\sigma(j)+1} )|^2\geq (1-\frac{m}{\gamma})\sum\limits^{m}_{i=1, i\, even} |y_{\overline{\sigma}(i)}|^2,
\]
where $\overline{\sigma}(i)=\sigma(i)$ if $\sigma(i)$ is even and $\sigma(i-1)$ otherwise.  

So the right hand-side of (\ref{sum}) is less or equal  than
\begin{align*}
&\limsup\limits_{n\to\infty} \frac{m!}{(2\pi)^{md}}\, n^{\frac{m(Hd-2)}{2}}  \int_{O_{m,1} \times O_{m,2}} \int_{T^{\sigma}_{\varepsilon,2}} \prod^{m/2}_{j=1}|\widehat{f}(y_{2j})|^2\\
& \qquad\times \exp\Big(-\frac{1}{2}\frac{1-\beta(n)}{1+\overline{\beta}(n)}\Var\big(\sum\limits^{m}_{i=1, i\, odd} y_i\cdot (X_{u_i}-X_{u_{i-2}})\big)\Big)\\
&\qquad\times \exp\Big(-\frac{1}{2}(1-\beta(n))(\alpha_2-\phi_1(n^{\lambda_0-\frac{1}{2}}))\sum\limits^{m}_{i=1, i\, even}|y_i|^2(\Delta u_i)^{2H}\Big) \\
&\qquad\times \exp\Big(-\frac{1}{2}\frac{1-\beta(n)}{1+\overline{\beta}(n)}\Var\big(\sum\limits^{m}_{i=1, i\, odd} \sum^{m}_{j=i} ( y_{\sigma(j)}-y_{\sigma(j)+1} )\cdot \big(X_{v_i}-X_{v_{i-2}})\big)\Big)\\
&\qquad\times \exp\Big(-\frac{1}{2}(1-\beta(n))(\alpha_2-\phi_1(n^{\lambda_0-\frac{1}{2}}))(1-\frac{m}{\gamma})\sum\limits^{m}_{i=1, i\, even} |y_{\overline{\sigma}(i)}|^2(\Delta v_i)^{2H}\Big)\, dy\, du\, dv.
\end{align*}

As a consequence,  $\limsup\limits_{n\to\infty}\sum\limits_{\sigma\in \mathscr{P}_1} \overline{I}^{n,\sigma}_{m,\gamma,3} $ is less than or equal to the product of 
\begin{align*}
 \left(\int^{+\infty}_0\int^{+\infty}_0\int_{\R^d}|\widehat{f}(z)|^2 \exp\left(-\frac{\alpha_2}{2}|z|^2(u^{2H}+(1-\frac{m}{\gamma})v^{2H})\right)\, dz\, du\, dv\right)^{m/2} 
 \end{align*}
 and
\begin{align} \label{lt1}
&\limsup\limits_{n\to\infty}\sum\limits_{\sigma\in \mathscr{P}_1} \frac{m!}{(2\pi)^{md}}\, n^{\frac{m(Hd-2)}{2}}  \int_{D^{m/2}_{odd,1}\times D^{m/2}_{odd, 2}} \int_{\R^{\frac{md}{2}}} \nn\\
& \qquad\times \exp\Big(-\frac{1}{2}\frac{1-\beta(n)}{1+\overline{\beta}(n)}\Var\big(\sum\limits^{m}_{i=1, i\, odd} y_i\cdot (X_{u_i}-X_{u_{i-2}})\big)\Big)\nn\\
&\qquad\times \exp\Big(-\frac{1}{2}\frac{1-\beta(n)}{1+\overline{\beta}(n)}\Var\big(\sum\limits^{m}_{i=1, i\, odd} \sum^{m}_{j=i} ( y_{\sigma(j)}-y_{\sigma(j)+1} )\cdot \big(X_{v_i}-X_{v_{i-2}})\big)\Big)\, d\overline{y}\, d\overline{u}\, d\overline{v}
\end{align}
where $D^{m/2}_{odd,\ell}=\{na_{\ell}<u_1<u_3<\dots<u_{m-1}<nb_{\ell}\}$ for $\ell=1,2$, $d\overline{y}=\prod\limits^m_{i=1,odd} dy_i$,$d\overline{u}=\prod\limits^m_{i=1,odd} du_i$ and $d\overline{v}=\prod\limits^m_{i=1,odd} dv_i$.

By  the self-similarity property and a change of variables,  and taking into account  Lemma \ref{lema2.1},  we can show that the
 $\limsup$ in (\ref{lt1}) is equal to $\frac{2^{m}}{(2\pi)^{\frac{md}{2}}}\E\big[\Delta_{E}\Lambda\big]^{m}$. Therefore, 
$\limsup\limits_{n\to\infty}\sum\limits_{\sigma\in \mathscr{P}_1} \overline{I}^{n,\sigma}_{m,\gamma,3}$ is less than
\begin{align*}
 \frac{2^{m}}{(2\pi)^{\frac{md}{2}}}\E\big[\Delta_{E}\Lambda\big]^{m} \left(\int^{+\infty}_0\int^{+\infty}_0\int_{\R^d}|\widehat{f}(z)|^2 \exp\left(-\frac{\alpha_2}{2}|z|^2(u^{2H}+(1-\frac{m}{\gamma})v^{2H})\right)\, dz\, du\, dv\right)^{m/2}.
\end{align*}

Using similar arguments as above,  $\liminf\limits_{n\to\infty}\sum\limits_{\sigma\in \mathscr{P}_1} \overline{I}^{n,\sigma}_{m,\gamma,3}$ is greater than
\begin{align*}
\frac{2^{m}}{(2\pi)^{\frac{md}{2}}}\E\big[\Delta_{E}\Lambda\big]^{m} \left(\int^{+\infty}_0\int^{+\infty}_0\int_{\R^d}|\widehat{f}(z)|^2 1_{\{|z|>\gamma \varepsilon\}} \exp\left(-\frac{\alpha_2}{2}|z|^2(u^{2H}+(1+\frac{m}{\gamma})v^{2H})\right)\, dz\, du\, dv\right)^{m/2}.
\end{align*}

Therefore,
\begin{align*}
\limsup\limits_{n\to\infty} I^n_m
&\leq \limsup\limits_{n\to\infty}\sum_{\sigma\in \mathscr{P}_1} \overline{I}^{n,\sigma}_{m,\gamma,3}+c_{22} \int_{\R^{2d}} |\widehat{f}(x)|^2 |\widehat{f}(y)|^2|x|^{-\frac{2}{H}}|y|^{-\frac{2}{H}}\, 1_{\{|x-y|\leq 4m\varepsilon\}}\, dx\, dy\\
&\qquad\qquad\qquad\qquad+c_{23}\int_{|y|\leq \gamma \varepsilon} |\widehat{f}(y)|^2 |y|^{-\frac{2}{H}}\, dy
\end{align*} 
and
\begin{align*}
\liminf\limits_{n\to\infty} I^n_m
&\geq \liminf\limits_{n\to\infty}\sum_{\sigma\in \mathscr{P}_1} \overline{I}^{n,\sigma}_{m,\gamma,3}-c_{22}\int_{\R^{2d}} |\widehat{f}(x)|^2 |\widehat{f}(y)|^2|x|^{-\frac{2}{H}}|y|^{-\frac{2}{H}}\, 1_{\{|x-y|\leq 4m\varepsilon\}}\\
&\qquad\qquad\qquad\qquad-c_{23}\int_{|y|\leq \gamma \varepsilon} |\widehat{f}(y)|^2 |y|^{-\frac{2}{H}}\, dy.
\end{align*}
Taking $\varepsilon\to 0$ first and then $\gamma\to\infty$, 
\[
\lim\limits_{n\to\infty} I^n_m=D^{m/2}_{f,H,d}\, \E\big[\Delta_{E}\Lambda\big]^{m}.
\]

Recall $\limsup\limits_{n\to\infty} \left|\E[F_{n}(a_1,b_1; a_2, b_2)^{m}]-I^n_{m} \right|=0$ in \textbf{Step1}.  We obtain the desired convergence of even moments.
\end{proof}

\section{Proof of Theorem \ref{thm1}}

Let
\begin{equation}
F_{n}(t_1, t_2)=n^{\frac{Hd-2}{2}} \int^{nt_1}_0\int^{nt_2}_0 f(X^{(1)}_u-X^{(2)}_v)\, du\, dv.   \label{t.e.1}
\end{equation}
The proof of Theorem \ref{thm1} will be done in two steps. We first show tightness and then establish the convergence of moments. Tightness will be deduced from the following result.
\begin{proposition} \label{tight}
For any integer $m\geq 1$ and any $0\leq a_1, a_2, b_1, b_2\leq T$,
\begin{equation*}
\E\Big[\big(F_{n}(b_1,b_2)-F_{n}(a_1,a_2)\big)^{2m}\Big] \leq C\,
\Big[(|b_1-a_1|+|b_2-a_2|)^{1-\frac{Hd}{2}} \, N(f)^2\Big]^{m}\,,  \label{eq1.2}
\end{equation*}
where $C$ is a positive constant depending only on $H$, $d$, $m$ and $T$.
\end{proposition}

\begin{proof}  Consider the decomposition
\begin{align*}
&F_{n}(b_1, b_2)-F_{n}(a_1, a_2)\\
&=n^{\frac{Hd-2}{2}}\int^{nb_1}_0\int^{nb_2}_0 f(X^{(1)}_u-X^{(2)}_v)\, du\, dv-n^{\frac{Hd-2}{2}} \int^{na_1}_0\int^{na_2}_0 f(X^{(1)}_u-X^{(2)}_v)\, du\, dv \\
&=n^{\frac{Hd-2}{2}}\int^{nb_1}_{na_1}\int^{nb_2}_{0} f(X^{(1)}_u-X^{(2)}_v)\, du\, dv+n^{\frac{Hd-2}{2}} \int^{na_1}_{0}\int^{nb_2}_{na_2} f(X^{(1)}_u-X^{(2)}_v)\, du\, dv \\
&=:I_1+I_2.
\end{align*}
So it suffices to show
\begin{equation}
\E\big(I_1^{2m}\big) \leq C\,
\Big[b_2^{1-\frac{Hd}{2}}|b_1-a_1|^{1-\frac{Hd}{2}} \, N(f)^2 \Big]^{m}\,.  \label{t.e.2}
\end{equation}
Note that
\begin{equation}
\E\big(I_1^{2m}\big)=n^{m(Hd-2)} \E \Big[\int_{D_0^{2m}}\prod_{i=1}^{2m} f(X^{(1)}_{u_i}-X^{(2)}_{v_i})\, du\, dv\Big], \label{t1.e.1}
\end{equation}
where $D_0=[na_1,nb_1]\times[0,nb_2]$.  

Using the arguments as in Subsection 3.1, we obtain
\begin{align*}
\E\big(I_1^{2m}\big)
&\leq c_1\, ( G_n(a_1,b_1))^{1/2}\, ( G_n(0,b_2))^{1/2},
\end{align*}
where 
\begin{align*}
G_n(a,b)
&=n^{m(Hd-2)}  \int_{\R^{md}} \int_{[na,nb]^{4m}}\prod _{i=1}^m| \widehat{f} (x_i-x_{i+1} ) |\\
&\qquad\qquad\qquad\qquad \times  \exp\bigg( -\frac {\kappa} 2   \sum_{i=1}^{m} |x_i|^{2H} (s_i^{2H} + r_i^{2H})\bigg) \,ds\, dr\, dx.
\end{align*}
By Lemma \ref{chain}, taking $\lambda =0$, we can write
\begin{align*}
G_n(0,b_2)\leq   c_2\, b_2^{m(2-Hd)}\, N(f)^{2m}
\end{align*}
and
\begin{align*}
G_n(a_1,b_1)\leq c_3\, (b_1-a_1)^{m(2-Hd)}\, N(f)^{2m}.
\end{align*}
Therefore,
\begin{align*}
\E\big(I_1^{2m}\big)
&\leq c_4\,b_2^{m(1-\frac{Hd}{2})}\, (b_1-a_1)^{m(1-\frac{Hd}{2})}\, N(f)^{2m}.
\end{align*}
This completes the proof.
\end{proof}

\medskip
In the remaining of this section, we prove that the moments of $F_n(t_1,t_2)$ converge to the
corresponding moments of $\Lambda(t_1,t_2)$. 

Fix a finite number of disjoint rectangles $E_i=(na_i, nb_i]\times(nc_i,nd_i]$ with
$i=1,\dots ,N$.  Let $\mathbf{m}=(m_1, \dots, m_N)$ be a fixed multi-index with $m_i\in\N$ for $i=1,\dots ,N$.
Set $\sum\limits_{i=1}^N m_i=|\mathbf{m}|$ and  $\prod\limits_{i=1}^N m_i!=\mathbf{m}!$. Recall the definition of $F_{n}(t_1, t_2)$ in \eref{t.e.1}. Let
\[
F_{n}(E_i)=n^{\frac{Hd-2}{2}}\int_{E_i} f(X^{(1)}_u-X^{(2)}_v)\, du\, dv.
\]
We need to consider the following sequence of random variables
 \[
        G_n=\prod_{i=1}^N \left(F_{n}(E_i) \right)^{m_i}
 \]
and compute  $\lim\limits_{n\rightarrow  \infty}  \E(G_n) $. Note that the expectation of $G_n$ can be written as 
\begin{align}  \label{EGn}
      \E(G_n) =n^{\frac{|\mathbf{m}|(Hd-2)} 2}  \E \bigg(
           \int_{D_\mathbf{m}}  \prod_{i=1}^N\prod_{j=1}^{m_i}f(X^{(1)}_{s^i_{j,1}}-X^{(2)}_{s^i_{j,2}})\, ds \bigg),
 \end{align}
where 
\begin{equation}
D_\mathbf{m}=\big\{s \in \R^{2|\mathbf{m}|}_+: na_i\leq s^i_{j,1}\leq nb_i, nc_i\leq s^i_{j,2}\leq nd_i,1\le i\le N, 1\leq j\leq m_i\big\}. \label{e.3.1.4}
\end{equation}

\medskip
To establish the convergence of moments, we need to consider two cases. One is that at least one of the exponents $m_i$ is odd. The other is that all exponents $m_i$ are even. We start with the convergence of moments in the first case.
\begin{proposition}  \label{odd} Suppose that at least one of the exponents $m_i$ is odd. Then
\begin{equation}
\lim_{n\to+\infty}\E(G_n)=0 \label{prop3.2}.
\end{equation}
\end{proposition}

\begin{proof} 
The proof will be done in several steps.

\medskip \noindent
\textbf{Step 1} \quad We subdivide the disjoint rectangles $\{E_i: 1\leq i\leq N\}$ in such a way that $\prod^N_{i=1} E^{m_i}_i$
can be represented as the union of sets of the form $\prod^L_{\alpha=1}\prod^M_{\beta=1}E^{m_{\alpha,\beta}}_{\alpha,\beta}$ where $E_{\alpha,\beta} =(na_\alpha, nb_\alpha] \times (nc_\beta, nd_\beta]$, $m_{\alpha,\beta}\in\N\cup\{0\}$ and $\sum^L_{\alpha=1}\sum^M_{\beta=1}m_{\alpha,\beta}=|\mathbf{m}|$.  We also impose  $a_\alpha \le b_{\alpha+1}$ and $c_\beta \le d_{\beta+1}$ for each $\alpha, \beta$.
From the assumption, we see that at least one of the exponents $m_{\alpha, \beta}$ is odd. Therefore, to prove \eref{prop3.2}, it suffices to consider the convergence of moments of $F_n(t_1,t_2)$ on disjoint rectangles $\{E_{\alpha,\beta} =(na_\alpha, nb_\alpha] \times (nc_\beta, nd_\beta]: \alpha=1,\dots, L; \beta =1,\dots, M\}$. In this case, $\mathbf{m}$ is the multi-index $(m_{1,1} ,\dots, m_{N,M})$ and $
D_{\mathbf{m}}=\prod^L_{\alpha=1}\prod^M_{\beta=1} E^{m_{\alpha,\beta}}_{\alpha,\beta}$. We denote the points in $E^{m_{\alpha,\beta}}_{\alpha,\beta}$ by $\prod^{m_{\alpha, \beta}}_{j=1}(s^j_{\alpha,\beta}, u^j_{\alpha, \beta})$. By tightness, we can assume that $b_{\alpha-1} +\delta \le a_{\alpha}$ and $d_{\beta-1} +\delta \le c_{\beta}$ for all $\alpha $ and $\beta$ and some $\delta>0$.

 With the above notation we need to estimate
 \begin{eqnarray*}
 \E (G_n) &=&
 \frac{n^{\frac{|\mathbf{m}|(Hd-2)}{2}}}{(2\pi)^{|\mathbf{m}|d}}    \int_{\R^{|\mathbf{m}|d}}  \int_{D_\mathbf{m}}
            \prod^{L}_{\alpha =1}\prod^{M}_{\beta =1}\prod^{m_{\alpha,\beta}}_{j=1} \widehat{f}(\xi^j_{\alpha,\beta})  \\
 & & \qquad \times \exp\bigg(-\frac{1}{2}\Var\Big( \sum^{L}_{\alpha =1} \sum_{\beta =1}^M\sum^{m_{\alpha,\beta} }_{j=1} \xi^j_{\alpha,\beta} \cdot \big(X^{(1)}_{s^j_{\alpha,\beta}} - X^{(2)}_{u^j_{\alpha,\beta}} \big)\Big)\bigg)  \, ds\,  du\, d\xi.
\end{eqnarray*}

For the simplicity of notation, we set
 \[
 \Phi_n(\xi)=\prod^{L}_{\alpha =1}\prod^{M}_{\beta =1}\prod^{m_{\alpha,\beta}}_{j=1} \big|\widehat{f}(\xi^j_{\alpha,\beta})\big|,
 \]
 and
 \[
I(\xi)=\int_{D_\mathbf{m}}  \exp\bigg( -\frac{1}{2}\Var\Big( \sum^{L}_{\alpha =1}\sum_{\beta=1}^M\sum^{m_{\alpha, \beta}}_{j=1} \xi^j_{\alpha,\beta}\cdot  \big(X^{(1)}_{s^j_{\alpha,\beta}}- X^{(2)}_{u^j_{\alpha,\beta}} \big)\bigg)\,  ds\, du.
\]
Then,
\begin{equation} \label{eqodd}
|\E (G_n)|\leq \frac{n^{\frac{|\mathbf{m}|(2-Hd)}{2}}}{(2\pi)^{|\mathbf{m}|d}}    \int_{\R^{|\mathbf{m}|d}}  \Phi_n(\xi)\,  I(\xi)\, d\xi.
\end{equation}

Let $\overline{m}_{\alpha}=\sum^M_{\beta=1} m_{\alpha,\beta}$ and $\underline{m}_{\beta}=\sum^L_{\alpha=1} m_{\alpha,\beta}$ for $\alpha=1,\dots, L$ and $\beta=1,\dots, M$.

\medskip \noindent
\textbf{Step 2} \quad  We first consider the case when one of $\overline{m}_{\alpha}$s or $\underline{m}_{\beta}$s is odd. Let
\[
I_1(\xi)=\int_{\prod\limits^L_{\alpha=1}\prod\limits^M_{\beta=1} (na_{\alpha}, nb_{\alpha}]^{m_{\alpha,\beta}}}  \exp\bigg( -\frac{1}{2}\Var\Big( \sum^{L}_{\alpha =1}\sum_{\beta=1}^M\sum^{m_{\alpha, \beta}}_{j=1} \xi^j_{\alpha,\beta}\cdot  X^{(1)}_{s^j_{\alpha,\beta}} \Big)\bigg)\,  ds
\]
and
 \[
I_2(\xi)=\int_{\prod\limits^L_{\alpha=1}\prod\limits^M_{\beta=1} (nc_{\beta}, nd_{\beta}]^{m_{\alpha,\beta}}}  \exp\bigg( -\frac{1}{2}\Var\Big( \sum^{L}_{\alpha =1}\sum_{\beta=1}^M\sum^{m_{\alpha, \beta}}_{j=1} \xi^j_{\alpha,\beta}\cdot  X^{(2)}_{u^j_{\alpha,\beta}}  \Big)\bigg)\, du.
\]
Then,
\begin{equation} \label{eqodd}
|\E (G_n)|\leq \frac{n^{\frac{|\mathbf{m}|(Hd-2)}{2}}}{(2\pi)^{|\mathbf{m}|d}}    \int_{\R^{|\mathbf{m}|d}}  \Phi_n(\xi)\,  I_1(\xi)\, I_2(\xi)\, d\xi.
\end{equation}

Applying Cauchy-Schwarz inequality to the right-hand  side of \eref{eqodd} gives
\[
|\E (G_n)|\leq  c_1   \bigg(n^{\frac{|\mathbf{m}|(Hd-2)}{2}} \int_{\R^{|\mathbf{m}|d}} \Phi_n(\xi)\,  \big(I_1(\xi)\big)^2 \, d\xi\bigg)^{\frac{1}{2}} \bigg(n^{\frac{|\mathbf{m}|(Hd-2)}{2}}\int_{\R^{|\mathbf{m}|d}}  \Phi_n(\xi)\,  \big(I_2(\xi)\big)^2\, d\xi\bigg)^{\frac{1}{2}}.
\]

With loss of generality, we can assume that one of  the $\overline{m}_{\alpha}$s is odd. Let $\alpha_0=\min\{\alpha: \overline{m}_{\alpha}\; \text{is odd}\}$, $m_1=\sum^{\alpha_0}_{\alpha=1}\overline{m}_{\alpha}$ and $m_2=\sum^{L}_{\alpha_0+1}\overline{m}_{\alpha}$. We are going to show
\begin{equation} \label{eqodd1}
\lim\limits_{n\to+\infty}n^{\frac{|\mathbf{m}|(Hd-2)}{2}} \int_{\R^{|\mathbf{m}|d}} \Phi_n(\xi)\,  \big(I_1(\xi)\big)^2 \, d\xi=0.
\end{equation}
Note that $m_1$ is odd and $\prod^L_{\alpha=1}\prod^M_{\beta=1} (na_{\alpha}, nb_{\alpha}]^{m_{\alpha,\beta}}\subseteq (na_1,nb_{\alpha_0}]^{m_1}\times(na_{\alpha_0+1},nb_L]^{m_2}$. It suffices to show \eref{eqodd1} in the case that $L=2$, $M=1$ and $m_{1,1}$ is odd. This follows easily from Lemma \ref{chain}.

If one of the $\underline{m}_{\beta}$s is odd, we can use the above argument to show that 
\[
\lim\limits_{n\to+\infty}n^{\frac{|\mathbf{m}|(Hd-2)}{2}} \int_{\R^{|\mathbf{m}|d}} \Phi_n(\xi)\,  \big(I_2(\xi)\big)^2 \, d\xi=0.
\]
If all of $\underline{m}_{\beta}$s are even, we can replace $\prod^L_{\alpha=1}\prod^M_{\beta=1} (nc_{\beta}, nd_{\beta}]^{m_{\alpha,\beta}}$ in the definition of $I_2(\xi)$ with $(nc_1,nd_M]^{|\mathbf{m}|}$
and then use the arguments in the proof of Proposition \ref{tight} to show that 
\[
n^{\frac{|\mathbf{m}|(Hd-2)}{2}} \int_{\R^{|\mathbf{m}|d}} \Phi_n(\xi)\,  \big(I_2(\xi)\big)^2 \, d\xi
\]
are uniformly bounded in $n$. Combining these results gives $\lim\limits_{n\to+\infty}|\E(G_n)|=0$ when one of the $\overline{m}_{\alpha}$s or $\underline{m}_{\beta}$s is odd. 

\medskip \noindent
\textbf{Step 3} \quad Recall the definitions of $\overline{m}_{\alpha}$ and $\underline{m}_{\beta}$. We now consider the case when all $\overline{m}_{\alpha}$ and $\underline{m}_{\beta}$ are even. We know that at least one of the exponents $m_{\alpha, \beta}$ is odd. Let $\beta_0=\min\{\beta: m_{\alpha, \beta}\; \text{is odd}\}$ and $\alpha_0=\min\{\alpha: m_{\alpha, \beta_0}\; \text{is odd}\}$. Then
\[
D_{\mathbf{m}}=\prod^L_{\alpha=1}\prod^M_{\beta=1} E^{m_{\alpha,\beta}}_{\alpha,\beta}\subseteq \prod^4_{i=1} E^{m_i}_i,
\]
where 
\[
 m_1=\sum^{\alpha_0}_{\alpha=1}\sum^{\beta_0}_{\beta=1}m_{\alpha,\beta},\, m_2=\sum^{\alpha_0}_{\alpha=1}\sum^{M}_{\beta=\beta_0+1}m_{\alpha,\beta},\, m_3=\sum^{L}_{\alpha=\alpha_0+1}\sum^{\beta_0}_{\beta=1}m_{\alpha,\beta},\, m_4=\sum^{L}_{\alpha=\alpha_0+1}\sum^{M}_{\beta=\beta_0+1}m_{\alpha,\beta},
\]
\[
E_1=(na_1,nb_{\alpha_0}]\times (nc_1,nd_{\beta_0}],\; E_2=(na_1,nb_{\alpha_0}]\times (nc_{\beta_0+1}, nd_M],
\]
and
\[
E_3=(na_{\alpha_0+1}, nb_L]\times (nc_1,nd_{\beta_0}],\, E_4=(na_{\alpha_0+1}, nb_L]\times (nc_{\beta_0+1}, nd_M]. 
\]
Therefore, it suffices to show
\begin{align} \label{oddzero}
\lim_{n\to+\infty}n^{\frac{|\mathbf{m}|(Hd-2)}{2}} \int_{\R^{2|\mathbf{m}|d}}  \Phi_n(\xi)\,  I(\xi)\, d\xi=0
\end{align}
in the case that $L=2$, $M=2$ and $m_{1,1}$ is odd. By the arguments in {\bf Step 2}, we can assume that $m_{1,2}$, $m_{2,1}$ and $m_{2,2}$ are odd.  Using the paring technique as in Section 2, we have that $\mathscr{P}_1=\emptyset$ and 
\begin{align*}
\limsup_{n\to+\infty} n^{\frac{|\mathbf{m}|(Hd-2)}{2}} \int_{\R^{2|\mathbf{m}|d}}  \Phi_n(\xi)\,  I(\xi)\, d\xi
&\leq c_2\int_{\R^{2d}} |\widehat{f}(x)|^2|\widehat{f}(y)|^2|x|^{-\frac{2}{H}}|y|^{-\frac{2}{H}}1_{\{|x-y|\leq 4\varepsilon\sum\limits^2_{i=1}\sum\limits^2_{j=1}m_{i,j}\}}\, dx\, dy,
\end{align*}
where $\varepsilon$ is an arbitrary constant in $(0,1)$.

Letting $\varepsilon\downarrow 0$ then gives the desired result (\ref{oddzero}). This completes the proof.
\end{proof}

\medskip
Next we consider the convergence of moments when all exponents $m_i$ are even. Recall the definition of $\E(G_n)$ in \eref{EGn}. We observe
 \[
      \E(G_n) = \mathbf{m}!\, n^{\frac{|\mathbf{m}|(Hd-2)} 2}  \E \bigg(
           \int_{\overline{D}_\mathbf{m}}  \prod_{i=1}^N\prod_{j=1}^{m_i}f(X^{(1)}_{s^i_{j,1}}-X^{(2)}_{s^i_{j,2}})\, ds \bigg),
 \]
where 
\begin{align*}
\overline{D}_\mathbf{m}&=\Big\{s \in \R_+^{2|\mathbf{m}|}: na_i\leq s^i_{j,1}\leq nb_i, nc_i\leq s^i_{j,2}\leq nd_i,  s^i_{1,1}<s^i_{2,1}<\dots<s^i_{m_i,1}, 1\le i\le N, 1\leq j\leq m_i\Big\}.
\end{align*}

\begin{proposition}  \label{even} Suppose that all exponents $m_i$ are even. Then
\begin{equation} \label{prop3.3}
\lim_{n\to+\infty}\E(G_n)=D_{f,H,d}^{\frac{|\mathbf{m}|}2}\, \E \Big(\prod_{i=1}^N \big[\Delta_{E_i}\Lambda\big]^{m_i} \Big).
  \end{equation}
\end{proposition}

\begin{proof}   We subdivide the  disjoint rectangles $\{E_i: 1\leq i\leq N\}$ in such a way that $\prod^N_{i=1} E^{m_i}_i$
can be represented as the union of sets of the form $\prod^L_{\alpha=1}\prod^M_{\beta=1}E^{m_{\alpha,\beta}}_{\alpha,\beta}$ where $E_{\alpha,\beta} =(na_\alpha, nb_\alpha] \times (nc_\beta, nd_\beta]$, $m_{\alpha,\beta}\in\N\cup\{0\}$ and $\sum^L_{\alpha=1}\sum^M_{\beta=1}m_{\alpha,\beta}=|\mathbf{m}|$.   We also assume  $a_\alpha \le b_{\alpha+1}$ and $c_\beta \le d_{\beta+1}$ for each $\alpha, \beta$.By Proposition \ref{odd}, we can assume that all exponents $m_{\alpha, \beta}$ are even. Therefore, to prove \eref{prop3.3}, it suffices to consider the convergence of moments of $F_n(t_1,t_2)$ on disjoint rectangles $\{E_{\alpha,\beta} =(na_\alpha, nb_\alpha] \times (nc_\beta, nd_\beta]: \alpha=1,\dots, L; \beta =1,\dots, M\}$. In this case, $\mathbf{m}$ is the multi-index $(m_{1,1} ,\dots, m_{N,M})$ and $
D_{\mathbf{m}}=\prod^L_{\alpha=1}\prod^M_{\beta=1} E^{m_{\alpha,\beta}}_{\alpha,\beta}$. We denote the points in $E^{m_{\alpha,\beta}}_{\alpha,\beta}$ by $\prod^{m_{\alpha, \beta}}_{j=1}(s_{\alpha,\beta,j}, u_{\alpha, \beta,j})$. By tightness, we can assume that $b_{\alpha-1} +\delta \le a_{\alpha}$ and $d_{\beta-1} +\delta \le c_{\beta}$ for all $\alpha $ and $\beta$ and some $\delta>0$. 

 With the above notation we need to compute the limit of 
 \begin{align} \label{gn1}
 \E (G_{n,1}) &=
 \frac{n^{\frac{|\mathbf{m}|(Hd-2)}{2}}}{(2\pi)^{|\mathbf{m}|d}}    \int_{\R^{|\mathbf{m}|d}}  \int_{\prod^L_{\alpha=1}\prod^M_{\beta=1} E^{m_{\alpha,\beta}}_{\alpha,\beta}}
            \prod^{L}_{\alpha =1}\prod^{M}_{\beta =1}\prod^{m_{\alpha,\beta}}_{j=1} \widehat{f}(\xi_{\alpha,\beta,j})  \nonumber\\
 &  \qquad \times \exp\bigg(-\frac{1}{2}\Var\Big( \sum^{L}_{\alpha =1} \sum_{\beta =1}^M\sum^{m_{\alpha,\beta} }_{j=1} \xi_{\alpha,\beta,j} \cdot \big(X^{(1)}_{s_{\alpha,\beta,j}} - X^{(2)}_{u_{\alpha,\beta,j}} \big)\Big)\bigg)  \, ds\,  du\, d\xi.
\end{align}

We can first put $s_{\alpha,\beta,j}$ and $u_{\alpha,\beta,j}$ with $\alpha=1,\dots, L,\, \beta=1,\dots, M,\, j=1,\dots, m_{\alpha,\beta}$ in the increasing order, respectively.  Let $\overline{m}_{\alpha}=\sum^M_{\beta=1} m_{\alpha,\beta}$ and $\underline{m}_{\beta}=\sum^L_{\alpha=1} m_{\alpha,\beta}$ for $\alpha=1,\dots, L$ and $\beta=1,\dots, M$. We see that there are only $\prod^{L}_{\alpha=1} \overline{m}_{\alpha}!$  permutations for $s_{\alpha,\beta,j}$ and $\prod^{M}_{\beta=1} \underline{m}_{\beta}!$ permutations for $u_{\alpha,\beta,j}$. So the total number of permutations is  $\prod^{L}_{\alpha=1} \overline{m}_{\alpha}!\prod^{M}_{\beta=1} \underline{m}_{\beta}!$. However, not all these permutations will contribute to the limit of $\E (G_{n,1})$ as $n\to\infty$. According to the paring technique, we will see that there are only $\prod^{L}_{\alpha=1} \prod^{M}_{\beta=1} 2^{\frac{m_{\alpha,\beta}}{2}} (\frac{m_{\alpha,\beta}}{2})!$ of them contributing to $\lim\limits_{n\to\infty}  \E (G_{n,1}) $ and they contribute equally.

Set $D_{L, M}=\{(\alpha,\beta,j): \alpha=1,\dots, L,\, \beta=1,\dots, M,\, j=1,\dots m_{\alpha,\beta}\}$. For points 
\[
\prod^L_{\alpha=1}\prod^M_{\beta=1}\prod^{m_{\alpha, \beta}}_{j=1}(s_{\alpha,\beta,j}, u_{\alpha, \beta,j})\in \prod^L_{\alpha=1}\prod^M_{\beta=1} E^{m_{\alpha,\beta}}_{\alpha,\beta},
\]
we let $\underline{\mathscr{P}}$ be the set of all bijections $\underline{R}$ from $D_{L,M}$ to $\{1,\dots, |\mathbf{m}|\}$ such that $\underline{R}(\alpha,\beta,j)<\underline{R}(\alpha',\beta',j')$ if and only if $s_{\alpha,\beta,j}<s_{\alpha',\beta',j'}$,  $\overline{\mathscr{P}}$ the set of all bijections $\overline{R}$ from $D_{L,M}$ to $\{1,\dots, |\mathbf{m}|\}$ such that $\overline{R}(\alpha,\beta,j)<\overline{R}(\alpha',\beta',j')$ if and only if $u_{\alpha,\beta,j}<u_{\alpha',\beta',j'}$.

Define $\mathscr{P}=\{(\underline{R},\overline{R}): \underline{R}\in \underline{\mathscr{P}},\, \overline{R}\in\overline{\mathscr{P}} \}$ and $\mathscr{P}_1$ to be the subset of $\mathscr{P}$ with its elements $(\underline{R},\overline{R})$ satisfying  
\begin{align} \label{pmt}
&\Big\{\{\underline{R}^{-1}(2i), \underline{R}^{-1}(2i-1)\}: i=1,\dots, |\mathbf{m}|/2 \Big\} \nonumber\\
&\qquad\qquad\qquad = \Big\{\{\overline{R}^{-1}(2i), \overline{R}^{-1}(2i-1)\}: i=1,\dots, |\mathbf{m}|/2 \Big\},
\end{align}
where $\underline{R}^{-1}$ and $\overline{R}^{-1}$ are inverses of $\underline{R}$ and $\overline{R}$, respectively.

For any $(\underline{R},\overline{R})\in\mathscr{P}$, define 
\begin{align*}
I^n_{\underline{R},\overline{R}} \label{INRR}
&= \frac{n^{\frac{|\mathbf{m}|(Hd-2)}{2}}}{(2\pi)^{|\mathbf{m}|d}} \int_{\R^{|\mathbf{m}|d}}  \int_{D_\mathbf{m}}
            \prod^{|\mathbf{m}|}_{\ell=1} \widehat{f}\big(\xi_{\underline{R}^{-1}(\ell)}\big)  \nonumber \\
 &\qquad\times \exp\bigg(-\frac{1}{2}\Var\Big( \sum^{|\mathbf{m}|}_{\ell=1} \xi_{\underline{R}^{-1}(\ell)} \cdot X^{(1)}_{s_{\underline{R}^{-1}(\ell)}} \Big)-\frac{1}{2}\Var\Big( \sum^{|\mathbf{m}|}_{\ell=1} \xi_{\overline{R}^{-1}(\ell)} \cdot X^{(2)}_{u_{\overline{R}^{-1}(\ell)}} \Big)\bigg)  \, ds\,  du\, d\xi.
\end{align*}
Then
 \begin{align*}
\E (G_{n,1}) &=\sum_{(\underline{R},\overline{R})\in\mathscr{P}} I^n_{\underline{R},\overline{R}}.
\end{align*}

For any $(\underline{R},\overline{R})\in\mathscr{P}-\mathscr{P}_1$,  using the paring technique as in the proof of Proposition \ref{prop2.7}, 
 \begin{align*}
\lim_{n\to+\infty} |I^n_{\underline{R},\overline{R}}|=0.
\end{align*}
That is, 
 \begin{align*}
\limsup_{n\to+\infty} |\E (G_{n,1})-\sum_{(\underline{R},\overline{R})\in\mathscr{P}_1} I^n_{\underline{R},\overline{R}}|=0.
\end{align*}

So we only need to find the limit of $\sum\limits_{(\underline{R},\overline{R})\in\mathscr{P}_1} I^n_{\underline{R},\overline{R}}$ as $n$ tends to $+\infty$. For any $(\underline{R},\overline{R})\in\mathscr{P}_1$, we  claim that (a) $\underline{R}^{-1}(2i)$ and $\underline{R}^{-1}(2i-1)$ only differ in the last element for all $i=1,\dots, |\mathbf{m}|/2$; (b) $\overline{R}^{-1}(2i)$ and $\overline{R}^{-1}(2i-1)$ only differ in the last element for all $i=1,\dots, |\mathbf{m}|/2$. This claim will be proved by induction. When $M=1$, using the ordering of $u_{\alpha,\beta,j}$s and (\ref{pmt}), we can easily see that the claim is true. Assume that the claim is true when $M=M_0$. We only need to show that it is also true when $M=M_0+1$. In this case, for elements $(s_{\alpha,\beta,j}, u_{\alpha,\beta,j})$ in $E^{m_{\alpha,\beta}}_{\alpha,\beta}$ with $\alpha=1,\dots, L$, $\beta=2,\dots, M_0+1$ and $j=1,\dots, m_{\alpha,\beta}$, by induction, the claim is true. Recall that $u_{\alpha', 1,j'}$s are strictly less than $u_{\alpha, \beta,j}$s whenever $\beta\geq 2$. This implies that the claim is true for all elements $(s_{\alpha,\beta,j}, u_{\alpha,\beta,j})$ in $E^{m_{\alpha,\beta}}_{\alpha,\beta}$ with $\alpha=1,\dots, L$, $\beta=1,\dots, M_0+1$ and $j=1,\dots, m_{\alpha,\beta}$. Therefore, the total number of elements in $\mathscr{P}_1$ is  $\prod^{L}_{\alpha=1} \prod^{M}_{\beta=1} 2^{\frac{m_{\alpha,\beta}}{2}} (\frac{m_{\alpha,\beta}}{2})!=2^{\frac{|\mathbf{m}|}{2}}\prod^{L}_{\alpha=1} \prod^{M}_{\beta=1}  (\frac{m_{\alpha,\beta}}{2})!$. 

Thanks to the claim above,  using similar arguments as in the proof of Proposition \ref{prop2.7} and then Lemma \ref{lema2.1}, we  obtain that 
\[
\lim_{n\to+\infty}\sum\limits_{(\underline{R},\overline{R})\in\mathscr{P}_1} I^n_{\underline{R},\overline{R}}= (D_{f,H,d})^{\frac{|\mathbf{m}|}{2}}   \E\Big(\prod^{L}_{\alpha=1} \prod^{M}_{\beta=1} [\Delta_{E_{\alpha,\beta}}\Lambda]^{m_{\alpha,\beta}} \Big).
\]
Therefore, 
\[
\lim_{n\to+\infty} \E (G_{n,1})= (D_{f,H,d})^{\frac{|\mathbf{m}|}{2}}   \E\Big(\prod^{L}_{\alpha=1} \prod^{M}_{\beta=1} [\Delta_{E_{\alpha,\beta}}\Lambda]^{m_{\alpha,\beta}} \Big).
\]
Recall the definition of $ \E (G_{n,1})$ in (\ref{gn1}) and the fact that $\prod^N_{i=1} E^{m_i}_i$
is represented as the union of sets of the form $\prod^L_{\alpha=1}\prod^M_{\beta=1}E^{m_{\alpha,\beta}}_{\alpha,\beta}$ with $\sum^L_{\alpha=1}\sum^M_{\beta=1}m_{\alpha,\beta}=|\mathbf{m}|$. We have
\[
\lim_{n\to+\infty} \E(G_n)= (D_{f,H,d})^{\frac{|\mathbf{m}|}{2}}   \E \Big(\prod_{i=1}^N \big[\Delta_{E_i}\Lambda\big]^{m_i} \Big).
\]
This gives the desired result. \end{proof}

\bigskip

\noindent
{\bf Proof of Theorem \ref{thm1}}:  This follows from Propositions \ref{tight}, \ref{odd} and \ref{even}.

\bigskip
\bigskip

$\begin{array}{cc}
\begin{minipage}[t]{1\textwidth}
{\bf David Nualart}\\
Department of Mathematics, University of Kansas, Lawrence, Kansas 66045, USA\\
\texttt{nulart@math.ku.edu}
\end{minipage}
\hfill
\end{array}$

\medskip

$\begin{array}{cc}
\begin{minipage}[t]{1\textwidth}
{\bf Fangjun Xu}\\
School of Statistics, East China Normal University, Shanghai 200241, China \\
NYU-ECNU Institute of Mathematical Sciences, NYU Shanghai, Shanghai 200062, China\\
\texttt{fangjunxu@gmail.com, fjxu@finance.ecnu.edu.cn}
\end{minipage}
\hfill
\end{array}$


\begin{thebibliography}{99}

\bibitem{Be} S. M. Berman:  Local nondeterminism and local times of Gaussian processes.
\textit{Indiana University Mathematics Journal}, {\bf 23},  69--94, 1973.



\bibitem{bx} J. Bi and F. Xu: A first-order limit law for functionals of two independent fractional Brownian motions in the critical case. \textit{J. Theor.  Probab.}, \textbf{29}, 941--957, 2016.

\bibitem{biane} P. Biane:  Comportement asymptotique de certaines fonctionnelles
additives de plusieurs mouvements browniens. \textit{S\'{e}minaire
de Probabilit\'{e}s, XXIII}, Lecture Notes in Math., \textbf{1372},
Springer, Berlin, 198--233, 1989.

\bibitem{Bojdecki}
T. Bojdecki, L. Gorostiza and A. Talarczyk: Sub-fractional Brownian motion and its relation to occupation times.   
{\it Statist. Probab. Letters}, {\bf  69}, 405--419, 2004.

\bibitem{HN} D. Harnett and D. Nualart: Central limit theorem for functionals of a generalized self-similar Gaussian process. \textit{Stoch. Proc. Appl.}, \textbf{128}, 404--425, 2018.


\bibitem{Houdre}
C. Houdr\'e and J. Villa:  An example of infinite dimensional quasi-helix.  \textit{ Contemp. Math.},  {\bf 336},  3--39, 2003.

\bibitem{hnx} Y. Hu., D. Nualart and F. Xu: Central limit theorem for an additive functional of the fractional Brownian motion. \textit{Ann. Prob.}, \textbf{42}, 168--203, 2014.

\bibitem{LeGall} J. F. Le Gall:  Sur la saucisse de Wiener et les points multiples du Mouvement Brownien. \textit{Ann. Prob.}, \textbf{14}, 1219--1244, 1986.

\bibitem{nol} D. Nualart and S. Ortiz-Latorre: Intersection local time for two independent fractional Brownian motions. \textit{J. Theor. Probab.}, \textbf{20}, 759--757, 2007.

\bibitem{nx1}  D. Nualart and F. Xu: Central limit theorem for an additive functional of the fractional Brownian motion II. \textit{Electron. Commun. Probab.}, \textbf{18}(74), 1--10, 2013.

\bibitem{nx2}  D. Nualart and F. Xu: Central limit theorem for functionals of two independent fractional Brownian motions. \textit{Stoch. Proc. Appl.}, \textbf{124}, 3782--3806, 2014.

\bibitem{Chavez}
J. Ruiz de Chavez and C. Tudor: A decomposition of sub-fractional Brownian motion.   \textit{Math. Reports}, {\bf11}, 67--74, 2009.


\bibitem{RussoTudor06}
F. Russo and C.A. Tudor: On bifractional Brownian motion.  \textit{ Stoch. Proc. Appl.}, {\bf 116},  830--856, 2006.


\bibitem{sxy} J. Song, F. Xu and Q. Yu:  Limit theorems for functionals of two independent Gaussian processes. arXiv:1711.10642v2.

\bibitem{weinryb_yor} S. Weinryb, M. Yor: Le mouvement Brownien de L\'{e}vy index\'{e} par $\R^3$ comme limite centrale de temps locaux d'intersection. S\'{e}minaire de Probabilit\'{e}s XXII, Lecture notes in Mathematics \textbf{1321}, 225--248, 1988.
 
\bibitem{wu_xiao} D. Wu and Y. Xiao: Regularity of Intersection Local Times of Fractional Brownian Motions. \textit{J. Theor.  Probab.}, \textbf{23}, 972--1001, 2010.

\bibitem{xu} F. Xu: Second order limit laws for occupation times of the fractional Brownian motion. \textit{J. Appl. Prob.}, \textbf{54}, 444--461, 2017.

\end{thebibliography}
\end{document}